\documentclass[12pt]{article}
\usepackage{graphicx}
\usepackage[T1]{fontenc}
\usepackage[cp1250]{inputenc}
\usepackage[top=15mm,bottom=15mm,left=23mm,right=23mm]{geometry}
\usepackage{amsmath,amstext,amssymb,amsopn,color,amsthm,textcomp}
\usepackage[dvipsnames]{xcolor}
\usepackage{eucal,mathrsfs,dsfont}
\usepackage{enumerate}

\usepackage[hidelinks]{hyperref}
\hypersetup{colorlinks=false}

\theoremstyle{plain}
\newtheorem{thm}{Theorem}[section]

\newtheorem{lemma}[thm]{Lemma}
\newtheorem{prop}[thm]{Proposition}
\newtheorem{cor}[thm]{Corollary}
\newtheorem{problem}[thm]{Problem}

\theoremstyle{definition}
\newtheorem{defn}{Definition}[section]

\theoremstyle{remark}
\newtheorem*{rem*}{Remark}
\newtheorem{rem}{Remark}[section]
\newcommand{\eps}{\varepsilon}

\newcommand{\calA}{\mathcal{A}}

\newcommand{\calF}{\mathcal{F}}
\newcommand{\calG}{\mathcal{G}}

\newcommand{\calL}{\mathcal{L}}
\newcommand{\R}{\mathbb{R}}

\newcommand{\N}{\mathbb{N}}

\newcommand{\Rd}{\mathbb{R}^d}

\renewcommand{\leq}{\leqslant}
\renewcommand{\geq}{\geqslant}

\DeclareMathOperator{\dist}{dist}
\DeclareMathOperator{\supp}{supp}
\DeclareMathOperator{\tr}{tr}

\def\({\left(}
\def\){\right)}
\def\[{\left[}
\def\]{\right]}
\def\<{\langle}
\def\>{\rangle}

\title{On the Bernoulli free boundary problems for the half Laplacian and for the spectral half Laplacian}
\author{
\ Sven Jarohs\footnote{Goethe-Universit\"at Frankfurt, Germany, jarohs@math.uni-frankfurt.de.},
Tadeusz Kulczycki\footnote{Wroclaw University of Science and Technology, Poland, tadeusz.kulczycki@pwr.edu.pl \newline T. Kulczycki was supported by the National Science Centre, Poland, grant no. 2019/33/B/ST1/02494.},
Paolo Salani\footnote{DiMaI, Universit\`a di Firenze, Italy, paolo.salani@unifi.it}
}

\date{}

\begin{document}
\maketitle

\begin{abstract}
We study the exterior and interior Bernoulli problems for the half Laplacian and the interior Bernoulli problem for the spectral half Laplacian. We concentrate on the existence and geometric properties of solutions. Our main results are the following. For the exterior Bernoulli problem for the half Laplacian, we show that under starshapedness assumptions on the data the free domain is starshaped. For the interior Bernoulli problem for the spectral half Laplacian, we show that under convexity assumptions on the data the free domain is convex and we prove a Brunn-Minkowski inequality for the Bernoulli constant. For Bernoulli problems for the half Laplacian we use a variational approach, whereas for Bernoulli problem for the spectral half Laplacian we use the Beurling method based on subsolutions. 
\end{abstract}

\section{Introduction}

The classical {\em Bernoulli problem}  splits into two similar but different cases, namely {\em the exterior and the interior} Bernoulli problems. They are free boundary problems, both involving two nested bounded open sets $K$ and $D$, with $\overline{K}\subset D$, and the Newtonian potential $u$ of $K$ with respect to $D$, i.e. the solution to the following Dirichlet problem:
$$
\Delta u=0\quad \text{in $D\setminus \overline{K}$ with $u=0$ on $\partial D$, $u=1$ on $\partial K$}.
$$
Then they are formulated as follows:
\begin{enumerate}
\item (Exterior BP) given $K$ and $\lambda>0$, find $D$ so that $|\nabla u|=\lambda$ on $\partial D$;
\item (Interior BP)  given $D$ and $\lambda>0$, find $K$ so that $|\nabla u|=\lambda$ on $\partial K$.
\end{enumerate}

Both problems have been largely investigated since the pioneering work of Beurling \cite{B1957}, first for the Laplacian and then for the $p$-Laplacian and other operators (we refer for instance to the papers \cite{AC1981, A1989, FR1997, HS1997, HS20000, HS2000, CT2002, CM2008, BS2009, DK2010} and references therein). 
According to literature, the exterior problem is easier and it usually has a unique solution for every value of $\lambda>0$, under suitable geometric assumptions on the involved domains. The interior problem is instead in general much more delicate and there is a parameter's threshold for $\lambda$, the so called {\em Bernoulli constant}, which determines the existence of a solution.

In this paper, we discuss the Bernoulli problem for the {\em half Laplacian} $(-\Delta)^{1/2}$ (or the square root of the Laplacian) and the {\em spectral half Laplacian} $(-\Delta_D)^{1/2}$ (or the spectral square root of the Laplacian), where the latter depends on an open bounded set $D\subset \R^d$.

Bernoulli problems for fractional Laplacians $(-\Delta)^{\alpha/2}$ are relevant in classical physical models in mediums where long range interactions are present. For example, they are related to models involving traveling wave solutions for planar cracks. They are also related to the theory of semipermeable membranes, see e.g. \cite{DL1976}. Such problems have been addressed for the first time by Caffarelli, Roquejoffre, and Sire in \cite{CRS2010} and subsequently they have been intensively studied in recent years see e.g. \cite{SR2012, DSS2015, DSS2015b, SSS2014, EKPSS2019, FR2020}. However, in these papers mainly the regularity of the free boundary of variational solutions is studied. In our paper instead, we concentrate on the existence and geometric properties of solutions. In particular, we show that, under appropriate assumptions, some geometric properties of the given domain (namely, starshapedness and convexity) imply similar geometric properties of the free domain. For the interior problems we study the related Bernoulli constants, which determine existence of solutions, and in the case of the spectral half Laplacian we prove a Brunn-Minkowski inequality and some of its consequences.

As usual, for a smooth function $\phi:\R^d\to\R$ the half Laplacian is given by
\begin{equation}\label{defi-halflaplacian}
(-\Delta)^{1/2}\phi(x)= \calA_d \int_{\R^d}\frac{2\phi(x)-\phi(x+y)-\phi(x-y)}{|y|^{d+1}}\ dy,\quad x\in \R^d
\end{equation}
where $\calA_d=\Gamma((d+1)/2)\pi^{-(d+1)/2}/2$ is a normalization constant so that the Fourier-transform of this operator has the symbol $|\cdot|$, see e.g. \cite{K2017} and the references therein.
It is noteworthy that the half Laplacian is a nonlocal operator and the boundary conditions have to be considered on the complement of the domain of interest. Moreover, the boundary regularity is different from the Laplacian, so that the gradient of a solution in general does not exist up to the boundary (see e.g. \cite{K1997,RS2014}). The precise formulations of the exterior and interior Bernoulli problems for the half Laplacian are as follows.

\begin{problem}[The exterior Bernoulli problem for the half Laplacian]
\label{ebp}
Given an open nonempty and bounded set $K\subset \R^d$ and $\lambda>0$, find a function $u\in C(\R^d)$ such that $\overline{K}\subset \{u>0\}$, $\{u>0\}$ is bounded and of class $C^1$, 
\begin{equation}
\label{bernoulli-problem basis}
\left\{
\begin{aligned}
(-\Delta)^{1/2}u&=0 &&\text{in $\{u>0\}\setminus \overline{K}$,}\\
u&=1 &&\text{in $\overline{K}$,}\\
u&=0 &&\text{in $\R^d\setminus\{u>0\}$,}\\
\end{aligned}
\right.
\end{equation}
and 
\begin{equation}
\label{normal_exterior}
\lim\limits_{t\to 0^+} \frac{u(\theta+t\nu(\theta))}{\sqrt{t}} =\lambda,
\end{equation}
for all $\theta\in \partial \{u>0\}$, where $\nu(\theta)$ denotes the interior normal of $\{u>0\}$ at $\theta$.
\end{problem} 

\begin{rem}\label{rem:homogeneity-exterior}
Let $s>0$ and $K_s=sK$. Then, if $u$ is a solution of Problem \ref{ebp} for $K$ and $\lambda$, it is easily seen that
$u_s(x)=u(x/s)$ solves the same problem for $K_s$ and $\lambda/\sqrt{s}$ (clearly we have $\{u_s>0\}=s\{u>0\}$).
\end{rem}

\begin{problem}[The interior Bernoulli problem for the half Laplacian]
\label{ibp}
Given an open nonempty and bounded set $D\subset \R^d$ and $\lambda>0$, find a function $u\in C(\R^d)$  such that $0\leq u\leq 1$,  the set $\{u= 1\}$ is nonempty, contained in $D$ and of class $C^1$,
\begin{equation}
\label{bernoulli-problem basis interior}
\left\{
\begin{aligned}
(-\Delta)^{1/2}u&=0 &&\text{in $D\setminus\{u=1\}$,}\\
u&=0 &&\text{in $\R^d\setminus D$,}\\
\end{aligned}
\right.
\end{equation}
and 
\begin{equation}
\label{normal_interior}
\lim\limits_{t\to 0^+} \frac{u(\theta) - u(\theta+t\nu(\theta))}{\sqrt{t}} =\lambda,
\end{equation}
for all  $\theta\in \partial \{u= 1\}$, where $\nu(\theta)$ denotes the exterior normal of $\{u = 1\}$ at $\theta$.
\end{problem}

Differently from the half Laplacian appearing above, the definition of the spectral half Laplacian depends on the domain of interest. Given a bounded open nonempty set $D$, let $\lambda_k$, $k\in \mathbb{N}$ be the Dirichlet eigenvalues of the negative Laplacian $-\Delta$ in $D$ with corresponding eigenfunctions $\phi_k$. Then, for $u\in H^1_0(D)$  we define the operator
\begin{equation}\label{defi-spectralhalflaplacian}
(-\Delta_D)^{1/2}u=\sum_{k=1}^{\infty} \lambda_k^{1/2} u_k\phi_k,\quad\text{where}\quad u_k=\int_{D} u(x)\phi_k(x)\ dx
\end{equation}
and we call it {\em the spectral half Laplacian} of $u$ in $D$ (see e.g. \cite{CT2010,SV2014}). Here, $H^1_0(D)$ denotes the closure of $C^{\infty}_c(D)$ with respect to the norm $u\mapsto \|\nabla u\|_{L^2(D)}$. We emphasize that also the spectral half Laplacian is nonlocal, however with nonlocality restricted to the set $D$. We focus here on the interior BP.

\begin{problem}[The interior Bernoulli problem for the spectral half Laplacian]
\label{isbp}
Given an open nonempty and bounded set $D\subset \R^d$ and $\lambda>0$, find a function $u\in C(\overline{D})$ such that $0\leq u\leq 1$, the set $\{u= 1\}$ is nonempty, contained in $D$ and of class $C^1$,
\begin{equation}
\label{bernoulli-problem basis spectral interior}
\left\{
\begin{aligned}
(-\Delta_D)^{1/2}u&=0 &&\text{in $D\setminus\{u=1\}$,}\\
u&=0 &&\text{on $\partial D$,}\\
\end{aligned}
\right.
\end{equation}
and 
$$
\lim\limits_{t\to 0^+} \frac{u(\theta) - u(\theta+t\nu(\theta))}{\sqrt{t}} =\lambda\quad\text{ for all }\theta\in \partial \{u = 1\},
$$
where $\nu(\theta)$ denotes the exterior normal of $\{u= 1\}$ at $\theta$.
\end{problem} 

\begin{rem}\label{rem:homogeneity-interior}
Similarly to the exterior BP, we can easily observe a homogeneity property of solutions of Problems \ref{ibp} and \ref{isbp} with respect to the domain $D$. Precisely, let $s>0$ and $D_s=sD$. Then, if $u$ is a solution of Problem \ref{ibp} or Problem \ref{isbp} for $D$ and $\lambda$, it is easily seen that $u_s(x)=\,u(x/s)$
solves the same problem for $D_s$ and $\lambda/\sqrt{s}$ (clearly we have $\{u_s=1\}=s\{u=1\}$).
\end{rem}

We aim at giving a solution to Problems \ref{ebp}, \ref{ibp}, and \ref{isbp} and investigate properties of these solutions. Problems \ref{ebp} and \ref{ibp} are a kind of free boundary problem, which has been studied in \cite{CRS2010,DSS2015,DSS2015b,SR2012,SSS2014} and we make use of the regularity results for the free boundary shown in these papers. In order to apply these results and to find solutions to the Bernoulli problems with the half Laplacian, we study two suitable functionals whose constrained minimizers satisfy a \textit{localized} version of the above problems.
For this, recall the Sobolev spaces $H^1(\R^{d+1}) = \{U \in L^2(\R^{d+1}): \, \nabla U \in L^2(\R^{d+1}) \}$,
$$
H^{1/2}(\R^d)=\{u\in L^2(\R^d)\;:\; [u]_{1}<\infty\} \quad\text{with}\quad [u]_{1}=\iint_{\R^d\times\R^d}\frac{(u(x)-u(y))^2}{|x-y|^{d+1}}\ 
 \, dx \, dy,
$$
and the continuous trace operator $\tr:H^1(\R^{d+1})\to H^{1/2}(\R^d)$, where the trace is considered with respect to the boundary of $\R^{d+1}_+ :=\{x\in \R^{d+1}\;:\; x_{d+1}>0\}$. Denoting by $\calL_d$ the $d$-dimensional Lebesgue measure on $\R^d$, for a given $\lambda>0$ we connect minimizers subject to the constraint $\tr U=1$ on $K$ of \begin{equation}
\label{ebp functional}
E_{\lambda} :H^1(\R^{d+1})\to [0,\infty],\quad E_{\lambda}(U)=\int_{\R^{d+1}} |\nabla U(x)|^2 \, dx + \frac{\pi}{4}\lambda^2 \,\calL_d (\{\tr U > 0\})
\end{equation}
with solutions to Problem \ref{ebp} (where $\{\tr U > 0\}=D$), and minimizers, subject to the constraints $\tr U=0$ on $\R^d\setminus D$, of
\begin{equation}
\label{ibp functional}
I_{\lambda,D} :H^1(\R^{d+1})\to [0,\infty],\quad I_{\lambda,D}(U)=\int_{\R^{d+1}} |\nabla U(x)|^2 \, dx + \frac{\pi}{4}\lambda^2 \,\calL_d (\{\tr U < 1\}\cap D)
\end{equation}
to Problem \ref{ibp}. More precisely, we study the following varational problems.
\begin{problem}
\label{variational}
Given $\lambda>0$ and $K\subset \R^d$ open nonempty and bounded, find a minimizer $U\in H^1(\R^{d+1})$ of $E_\lambda$ subject to the constraint $\tr U=1$ on $K$.
\end{problem}

\begin{problem}
\label{variational interior}
Given $\lambda>0$ and $D\subset \R^d$ open nonempty and bounded, find a nontrivial minimizer $U\in H^1(\R^{d+1})$ of $I_{\lambda,D}$ subject to the constraint $\tr U=0$ on $\R^d\setminus D$.
\end{problem}

Note that instead of studying functionals in $H^1(\R^{d+1})$ one may study functionals in $H^{1/2}(\R^d)$. To see this, recall that if $u\in H^{1/2}(\R^d)$ and $U\in H^1(\R^{d+1})$ is given by the harmonic extension of $u$, that is,
\begin{equation}\label{extension}
U(x,y)=\int_{\R^d}P(x-z,|y|)u(z)\ dz, \quad\text{for $x\in \R^d$, $y \ne 0$}
\end{equation}
where 
\begin{equation}
\label{Poisson}
P(x,y)=2\calA_d\frac{y}{(|x|^2+y^2)^{\frac{d+1}{2}}},
\end{equation}
and evenly reflected across $\partial \R^{d+1}_+$, we have
\begin{equation}
\label{energy identity}
\int_{\R^{d+1}}|\nabla U(x,y)|^2 \, dx \, dy, =\calA_d[u]_1.
\end{equation}
We prove that minimizers of Problems \ref{variational} and \ref{variational interior} satisfy (\ref{extension}). Hence, instead of minimizers of $E_\lambda$, we may study minimizers of 
\begin{equation}
\label{ebp functional 2}
 e_{\lambda}: \, H^{1/2}(\R^d)\to[0,\infty],\quad e_{\lambda}(u) =\calA_d\,[u]_1 + \frac{\pi}{4}\lambda^2 \,\calL_d (\{u> 0\})
\end{equation}
and similarly, instead of minimizers of $I_{\lambda,D}$, we may study minimizers of 
\begin{equation}
\label{ibp functional 2}
i_{\lambda,D}: \,H^{1/2}(\R^d)\to[0,\infty],\quad i_{\lambda,D}(u) =\calA_d\,[u]_1 + \frac{\pi}{4}\lambda^2 \,\calL_d (\{ u< 1\}\cap D).
\end{equation}

Now we state our main results for exterior and interior Bernoulli problems for the half Laplacian. The first one concerns Problem \ref{ebp} and, in order to properly state it, we use the following notion of starshapedness: given $B\subset K\subset \R^d$, $K$ is called \textit{starshaped with center $B$}, if $K$ is starshaped with respect to every $x\in B$. 

\begin{thm}
\label{bernoulli-exterior}
\begin{enumerate}[(i)]
\item Let $K\subset \R^d$ be an open nonempty bounded set. Then there exists a unique solution $U$ of Problem \ref{variational}. Moreover, $U$ satisfies
\begin{enumerate}[(a)]
\item $0 \le U \le 1$,
\item  $U$ is $1/2$-H{\"o}lder continuous on any compact subset of $\R^{d+1}\setminus (\overline{K}\times\{0\})$,
\item $U$ is harmonic in the set $\{U>0\}\setminus (\overline{K}\times\{0\})$ and in the set $\R^{d+1}\setminus(\R^d\times\{0\})$,
\item $\{\tr U > 0\}$ is bounded.
\item $U$ can be represented in $\R^{d+1}_+$  by the harmonic extension of its trace $u:=\tr U\in H^{1/2}(\R^d)$, $U$ is even in $x_{d+1}$ variable and $u$ satisfies (\ref{bernoulli-problem basis}). For any $\theta\in \partial\{u>0\} \setminus \overline{K}$ such that the interior unit normal vector to $\{u>0\}$ at $\theta$ exists (\ref{normal_exterior}) holds.
\end{enumerate}

\item If additionally $K$  has a $C^2$ boundary then $\partial\{u > 0\}$ and $\overline{K}$ are disjoint and $U \in C(\R^{d+1})$.

\item If additionally $K$ has a $C^2$ boundary and it is starshaped with center $B=B_r^d(x_0)$ for some $r>0$ and $x_0\in K$, then all the superlevel sets of $U$ are starshaped with center $B_r^{d+1}((x_0,0))$, $\partial\{u>0\}$ is of class $C^{\infty}$, and $u$ is a solution of Problem \ref{ebp}.
\end{enumerate}
\end{thm}

As it turns out, the solution found in Theorem \ref{bernoulli-exterior} of Problem \ref{ebp} is indeed the unique solution, see Proposition \ref{uniqueness of ebp} below. If $K\subset \R^d$ is an open nonempty bounded set with $C^2$ boundary then using the regularity results obtained by De Silva and Savin \cite[Theorem 1.1]{DSS2015} (see also \cite[Theorem 1.1]{EKPSS2019}) and Theorem \ref{bernoulli-exterior} (ii) we obtain that for $d = 2$ (\ref{normal_exterior}) is satisfied for all $\theta\in \partial\{u>0\}$ and for $d \ge 3$ (\ref{normal_exterior}) is satisfied for all $\theta\in \partial\{u>0\}$ except a set of Hausdorff dimension $\leq d-3$. If additionally $K$ is starshaped with center $B$ then (\ref{normal_exterior}) is satisfied for all $\theta\in \partial\{u>0\}$ for any dimension. This means that geometric assumptions on $K$ in the third part of Theorem \ref{bernoulli-exterior} give the regularity of the {\it complete} free boundary (in general the free boundary can be split into its regular part and its singular part see e.g. \cite{DSS2015,EKPSS2019}).

\medskip

Our next result concerns the interior Bernoulli problem. 

For $D\subset \R^d$ open nonempty and bounded, we define the {\em Bernoulli constant} of $D$ for the half Laplacian as follows:
\begin{equation}
\label{nontrivial exists intro}
\Lambda(D):=\inf\Big\{\lambda>0\;:\;\text{Problem  \ref{variational interior} has a  nontrivial solution}\Big\}.
\end{equation}

Then we have the following.

\begin{thm}
\label{bernoulli-interior}
Let $D\subset \R^d$ be an open nonempty bounded set. Then $0<\Lambda(D)<\infty$ and the Problem \ref{variational interior} has a solution $V$ if and only if $\lambda\geq \Lambda(D)$. Moreover, $V$ satisfies
\begin{enumerate}[(a)]
\item $0 \le V \le 1$,
\item $V$ is $1/2$-H{\"o}lder continuous on any compact subset of $\R^{d+1}\setminus (D^c\times\{0\})$,
\item $V$ is harmonic in the set $\{V>0\}\setminus (D^c\times\{0\})$ and in the set $\R^{d+1}\setminus(\R^d\times\{0\})$.
\item $V$ can be represented in $\R^{d+1}_+$  by the harmonic extension of its trace $u:=\tr V\in H^{1/2}(\R^d)$, $V$ is even in $x_{d+1}$ variable and $u$ is a solution of \eqref{bernoulli-problem basis interior}. For any $\theta\in \partial\{u=1\} \setminus D^c$ such that the exterior unit normal vector to $\{u=1\}$ at $\theta$ exists (\ref{normal_interior}) holds.
\end{enumerate} 
\end{thm}

 Note that if $\{u = 1\}$ is nonempty, contained in $D$ and of class $C^1$ then $u$ is a solution of Problem \ref{ibp}.
The Bernoulli constant for the half Laplacian satisfies simple monotonicity and homogeneity properties, which we state explicitly for the sake of clarity.
\begin{prop}
\label{monotonicity Lambda}
The Bernoulli constant for the half Laplacian is monotone decreasing with respect to inclusion, i.e. if $D_1\subset D_2\subset \R^d$ are open nonempty bounded sets, then $\Lambda(D_1)\geq \Lambda(D_2)$. Moreover, $\Lambda$ is positively homogeneous of degree $-1/2$, that is for any open nonempty bounded set $D \subset \R^d$ and $s > 0$ we have
\begin{equation}\label{homoLambda}
\Lambda(sD)=s^{-1/2}\Lambda(D).
\end{equation}
\end{prop}

Similarly to \cite[Theorem 1.2]{DK2010} we have the following isoperimetric inequality for $\Lambda(D)$.

\begin{prop}
\label{iso inequality}
Let $D\subset \R^d$ be an open nonempty bounded set and let $B\subset \R^d$ be an open ball with the same measure as $D$. Then $\Lambda(D)\geq \Lambda(B)$ and equality only holds if $D$ is a ball up to a set of measure zero.
\end{prop}

We denote $B_r^d(x_0) = \{x \in \R^d: \, |x - x_0| < r\}$ for $x_0 \in \R^{d }$ and $r > 0$. We drop the index $d$ here, if the dimension of the ball follows from context. We obtain the following rough estimate on $\Lambda(B_r^d(0))$.
\begin{lemma}
\label{estimate bernoulli}
For any $r>0$ we have 
$$
\frac{2}{\sqrt{\pi}} \frac{\sqrt[4]{d}}{\sqrt{r}} 3^{-(d+2)/2}<\Lambda(B_r^d(0))<\frac{2}{\sqrt{\pi}}\frac{\sqrt{d}}{\sqrt{r}} 2^{(d+3)/2}
$$
\end{lemma}

The above results imply the following estimate of $\Lambda(D)$.
\begin{cor}
For any $r > 0$ and any open nonempty bounded set $D \subset \R^d$ we have
$$
\frac{2}{\sqrt{\pi}} \frac{ \sqrt[4]{d} \, \pi^{1/4}}{(\calL_d(D))^{1/(2d)} (\Gamma(d/2 + 1))^{1/(2d)} 3^{(d+2)/2}} \le \Lambda(D) \le \frac{2}{\sqrt{\pi}} \frac{\sqrt{d} 2^{(d+3)/2}}{\sqrt{r(D)}},
$$ 
where $r(D)$ is the inradius of $D$.
\end{cor}

Now we present results concerning Problem \ref{isbp}. It is well known that problems for the spectral half Laplacian $(-\Delta_D)^{1/2}$ may be translated to some local problem for Dirichlet Laplacian in $D \times \R$, see e.g. \cite{CT2010}. More formally if $D \subset \R^d$ is an open nonempty bounded convex set, $K \subset D$ is a compact set and $u \in C(\overline{D} \times \R)$ is a bounded function satisfying
\begin{equation}
\nonumber
\left\{\begin{array}{ll}
\Delta u=0\quad&\text{in }(D\times\R)\setminus(K\times\{0\})\\
u=0\quad&\text{on }\partial D\times\R\\
u=1\quad&\text{in }K\times\{0\}\\
\end{array}\right.
\end{equation} 
then $(-\Delta_D)^{1/2} u = 0$ on $D \setminus K$. Using this we study the localized version of Problem \ref{isbp} in $D \times \R$. We apply the Beurling method, see e.g.  \cite{B1957, HS2000}. We construct a family of subsolutions to the problem in $D \times \R$ and by taking a limiting procedure we find a solution to the problem in $D \times \R$. Then it turns out that the restriction of this solution to $D$ is a solution to Problem \ref{isbp}.

To state the result for Problem \ref{isbp} we need to introduce the following class of functions.

\begin{defn}
\label{classF}
Let $\lambda > 0$, $d \ge 2$ and $D \subset \R^d$ be an open, nonempty, bounded convex set. Let $K \subset D$ be a nonempty compact set and let $v_K: \overline{D}\times\R \to [0,1]$, be the solution of the following Dirichlet problem
\begin{equation}\label{problemv}
\left\{\begin{array}{ll}
\Delta v_K=0\quad&\text{in }(D\times\R)\setminus(K\times\{0\})\\
v_K=0\quad&\text{on }\partial D\times\R\\
v_K=1\quad&\text{in }K\times\{0\}\\
\end{array}\right.
\end{equation}
We say that $K\in\calF(D,\lambda)$ if
\begin{equation}\label{overdetermined}
\sup_{y\in D\setminus K} \frac{|v_{K}(y,0)-1|}{\delta_K^{1/2}(y)} \le \lambda,
\end{equation}
where $\delta_K(y):=\dist(y, K)$ is the distance to $K$.\\
With an abuse of notation, for a function $v: \overline{D}\times\R \to [0,1]$ we write $v\in\calF(D,\lambda)$ (and we say that $v$ is a {\em subsolution} of Problem \ref{isbp}) if $v=v_K$ for some $K\in\calF(D,\lambda)$, i.e. if $v$ is the solution of \eqref{problemv} for some nonempty compact set $K\subset D$ and it satisfies \eqref{overdetermined}.
\end{defn}
The above definition is inspired by Definition 2.1 in \cite{HS2000}: \eqref{overdetermined} is in a certain strong sense an analog of the condition $|\nabla v| \le \lambda$. Note that if $K\in\calF(D,\lambda)$ then $v_K \in C(\overline{D} \times \R)$. Indeed, $D$ is convex so $v_K$ is continuous on $\partial D \times \R$. By \eqref{overdetermined} $x \to v_K(x,0)$ is continuous on $\overline{D}$, which implies that $v_K \in C(\overline{D} \times \R)$.

Notice that,  since $K$ is compact, its distance from $\partial D$ must be positive and, if $x\in K$, then $B_r^d(x)\subset D$ for $r$ sufficiently small.

Assume that $v\in\calF(D,\lambda)$ By the fact that $v \equiv 0$ on $\partial D \times \R$, $v$ takes values in $[0,1]$ and by standard arguments we obtain that that $v(x) \to 0$ as $|x|\to+\infty$. With this decay property it can be easily shown (for instance, by the moving planes method) that $v$ satisfies:

\begin{enumerate}[(i)]
\item for any $x \in D$ and $y \in \R$ we have $v(x,-y) = v(x,y)$,
\item for any $x \in D$ if $y_2 > y_1 \ge 0$ then $v(x,y_2) \le v(x,y_1)$.
\end{enumerate}

Thus, it follows that $u:=v(\cdot,0)$ is in particular a solution to \eqref{bernoulli-problem basis spectral interior} with 
$$
\lim\limits_{t\to 0^+} \frac{u(\theta) - u(\theta+t\nu(\theta))}{\sqrt{t}} \leq \lambda
$$
for any $\theta\in \partial K$ for which an exterior normal $\nu(\theta)$ exists. Our main results concerning the Bernoulli problem for the spectral half Laplacian are the following.

\begin{thm}
\label{thm_int_spectral}
Let $d \ge 2$, $D \subset \R^d$ be a bounded open nonempty convex set, $\lambda > 0$ and suppose that $\calF(D,\lambda)$ is not empty. Then there exists a solution $u$ to the free boundary Problem \ref{isbp}. Moreover the set $\{u = 1\}$ is convex.
\end{thm}

Let $d \ge 2$ and $D \subset \R^d$ be an open, nonempty, bounded convex set. Let us define the Bernoulli constant of $D$ for the spectral half Laplacian as 
$$
\Lambda_S(D) = \inf\{\lambda>0: \, \calF(D,\lambda) \, \text{is not empty}\}.
$$ 
 $\Lambda_S(D)$ satisfies the same monotonicity and homogeneity properties as  $\Lambda(D)$, as explicitly stated in the following proposition.  
\begin{prop}\label{homolambdaprop}
Let $d \ge 2$. The Bernoulli constant for the spectral half Laplacian is monotone decreasing with respect to set inclusion and positively homogeneous of degree $-1/2$, that is:
\begin{enumerate}[(i)]
\item if $D_1\subset D_2\subset \R^d$ are open nonempty bounded sets, then $\Lambda_S(D_1)\geq \Lambda_S(D_2)$,
\item for any open nonempty bounded set $D \subset \R^d$ and $s > 0$ we have $\Lambda_S(sD)=s^{-1/2}\Lambda_S(D)$.
\end{enumerate}
\end{prop}

The proof of Theorem \ref{thm_int_spectral} implies that the found solution is in fact a solution of a slight modification of Problem \ref{isbp1}:

\begin{problem}
\label{isbp1}
Given an open nonempty and bounded set $D\subset \R^d$ and $\lambda>0$, find a function $u\in C(\overline{D})$ satisfying all the conditions stated in Problem \ref{isbp} and additionally 
$$
\sup_{y\in D\setminus K} \frac{|u(y,0)-1|}{\delta_K^{1/2}(y)} = \lambda,
$$
where $K = \{x \in \R^d: \, u(x,0) = 1\}$.
\end{problem}

For the latter problem, we can indeed state a sharper existence result than Theorem \ref{thm_int_spectral}.
\begin{prop}\label{existlambda_S}
Let $d \ge 2$ and $D \subset \R^d$ be a bounded open nonempty convex set. Then a solution of Problem \ref{isbp1} exists if and only if $\lambda\geq\Lambda_S(D)$.
\end{prop}
An interesting open question is whether there exists a solution of Problem \ref{isbp} which is not a solution of Problem \ref{isbp1}.
\medskip

To give a rough estimate of $\Lambda_S$, we can observe that, by (i) of Proposition \ref{homolambdaprop}, it follows that 
$$\Lambda_S(B_{R(D)}^d(0))\le\Lambda_S(D)\le \Lambda_S(B_{r(D)}^d(0))\,,$$ 
where $r(D)$ is the inradius of $D$ and $R(D)$ is half the diameter of $D$. 
Hence we can give an estimate of the Bernoulli constant from above, thanks to the following.
\begin{prop}
\label{Bernoulli_constant}
For $d\ge 2$ and $r > 0$ we have $\Lambda_S(B_{r}^d(0)) \le \tilde{c}_d/\sqrt{r}$,
where 
$$
\tilde{c}_d = \sqrt{2} \left(\frac{2\sqrt{2}\Gamma(\frac{d}{2})}{\sqrt{\pi}\Gamma(\frac{d-1}{2})}\right) 
\left(1 - \frac{1}{\pi 2^{d - 2}} \int_0^{1/3} \frac{(1 - 3 b)^{(d-2)/2}}{b^{1/2}(1+b)} \, db\right)^{-1}.
$$
\end{prop}
Notice that we have $\tilde{c}_2 = 6/\pi \approx 1.91$, $\tilde{c}_3 \approx 2.31$.

As a consequence, if $\lambda > \tilde{c}_d/\sqrt{r(D)}$ (where $r(D)$ is the inradius of $D$) then the free boundary Problem \ref{isbp} has a solution. 
In the opposite direction, we may give an estimate from below of $\Lambda_S(D)$ in terms of the ball with the same diameter as $D$, as we have already said. But we can improve this estimate as a consequence of the following Brunn-Minkowski inequality.
\begin{thm}\label{BMlambda}
The Bernoulli constant for the spectral half Laplacian satisfies the following Brunn-Minkowski inequality: let $d \ge 2$, $D_0,\,D_1 \subset \R^d$ be open bounded nonempty convex sets and $s\in(0,1)$, then
\begin{equation}\label{BMlambdaeq}
\Lambda_S((1-s)D_0+sD_1)\leq \left[(1-s)\Lambda_S(D_0)^{-2}+s\Lambda_S(D_1)^{-2}\right]^{-1/2}\,.
\end{equation}
%Equality holds if and only if $D_0$ and $D_1$ are homothetic.
\end{thm}
Here $A+B=\{x+y\,:\,x\in A,\,y\in B\}$ denotes the Minkowski addition of sets.
Then, in other words, the above theorem states that, as a functional on the class of convex bodies, $\Lambda_S^{-2}$ is concave with respect to Minkowski addition.

We recall that the classical Brunn-Minkowski inequality was born in the framework of convex geometry and then extended to Lebesgue measurable sets: it states that the Lebesgue measure in $\R^d$ raised to power $1/d$ is concave with respect to Minkowski addition. It is a powerful inequality, at the core of the Brunn-Minkowski theory of convex bodies, it is strongly connected to many other important inequalities (in particulary to the isoperimetric inequality) and the related research is still very active (see for instance \cite{FMP, Christ1, Christ2, FJ1, FJ2} and the beautiful survey paper by Gardner \cite{Gardner} for more references). Similar inequalities hold for other geometric quantities, like perimeter or quermassintegrals of convex bodies, furthermore there is a functional version (namely, the Borell-Brascamp-Lieb inequality, see \cite{Borell1, BBL}) and, in recent years, Brunn-Minkowski inequalities have been proved for several functionals from calculus of variations (see for instance \cite{Borell2, Gardner, CS2003-2}), in particular, in connection with the present paper, for the Bernoulli constant \cite{BS2009} and for the $1$-Riesz capacity \cite{NR2015}.

As it is now well known, every Brunn-Minkowski inequality yields an Urysohn's type inequality by a standard procedure. Then, as a corollary of the previous theorem, we get the following.
\begin{cor}\label{urysohncor}
If $d \ge 2$, $D \subset \R^d$ is an open nonempty bounded convex set, then
$$
\Lambda_S(D)\geq \Lambda_S(B)\,,
$$
where $B$ is a ball with the same mean width of $D$. 
%Equality holds if and only if $D$ is a ball.
\end{cor}
We recall that the mean width $w(D)$ of a convex set $D\subset\R^d$ is defined as follows
$$
w(D)=\frac{2}{H^{d-1}(S^{d-1})}\int_{S^{d-1}}h_D(\xi)H^{d-1}(d\xi)\,
$$
where $h_D$ denotes the {\em support function} of $D$, i.e.
$$
h_D(y)=\sup_{x\in D}\langle x,y\rangle\quad\text{for }y\in\R^d\,,
$$
$H^{d-1}$ is the $(d-1)$-dimensional Hausdorff measure and $S^{d-1}$ is the unit sphere in $\R^d$.
Notice that $h_D$ is a $1$-homogeneous convex function in $\R^d$: for a given direction $\xi\in S^{d-1}$, it represents the signed distance  of the support hyperplane to $D$ with exterior unit normal $\xi$ from the origin and the width of the set $D$ in direction $\xi$ is given by $|h_D(\xi)-h_D(-\xi)|$. Clearly, the mean width is never greater than the diameter (indeed, it is strictly smaller unless D is a ball).
In the plane, $w(D)$ is just a multiple of the perimeter, exactly we have $w(D)=|\partial D|/\pi$. Then Corollary \ref{urysohncor}, for $d=2$, can be rephrased as follows: {\em among convex planar sets with given perimeter, balls have the smallest Bernoulli constant for the spectral half Laplacian}.

We adopt the convention that constants denoted by $c, c_0, c_1, \ldots$ may change their value from one use to the next. On the other hand, constants denoted by $C_0, C_1, \ldots$ do not change their value in the whole paper.

The paper is organized as follows. In Section 2 we study the exterior Bernoulli problem for the half Laplacian, in Section 3 the interior Bernoulli problem for the half Laplacian and in Section 4 the interior Bernoulli problem for the spectral half Laplacian. Section 5 is devoted to Brunn-Minkowski inequality for Bernoulli constant for the spectral half Laplacian. In Appendix we give a proof of an auxiliary, technical lemma, concerning $H^{1/2}(\R^d)$.

We finish this introduction with some remarks concerning other fractional Laplacians and the exterior Bernoulli problem for the spectral half Laplacian. One may ask why we study Bernoulli problems for the half Laplacian $(-\Delta)^{1/2}$ and the spectral half Laplacian $(-\Delta_D)^{1/2}$ and not for all fractional Laplacians $(-\Delta)^{\alpha/2}$ and spectral fractional Laplacians $(-\Delta_D)^{\alpha/2}$. The reason is that for $\alpha = 1$ there are known some important results, which are not available for other $\alpha$. In particular, in our paper in studying the exterior Bernoulli problem for $(-\Delta)^{1/2}$ we use Theorem 1.2 from \cite{DSS2015}, which roughly speaking states that a Lipschitz free boundary is smooth. Such a result is known only for $(-\Delta)^{1/2}$ and not for other fractional Laplacians. Similarly, in studying the interior Bernoulli problem for $(-\Delta_D)^{1/2}$ we use some deep results about harmonic functions from \cite{CS2003} to show that there exists a solution with convex free set. We do not know how to generalize it to other spectral fractional Laplacians. One may also ask why we did not study the exterior Bernoulli problem for the spectral half Laplacian. It turns out that for this problem the condition on the free boundary is not of the type (\ref{normal_exterior}) but of the type $|\nabla u| = \lambda$. So, in some sense, this problem has completely different nature than problems we study in our paper.

\section{The exterior Bernoulli problem for the half Laplacian}

\begin{lemma}
\label{general properties}
Let $K\subset \R^d$ be open nonempty and bounded, $\lambda>0$, and let $U\in H^1(\R^{d+1})$ be a minimizer of $E_{\lambda}$ subject to the constraint $\tr U=1$ on $K$ Then
\begin{enumerate}[(i)]
\item\label{item:d+1 1} $0 \le U \le 1$,
\item\label{item:d+1 2} $U$ is $1/2$-H{\"o}lder continuous on any compact subset of $\R^{d+1}\setminus (\overline{K}\times\{0\})$,
and
\item\label{item:d+1 3} $U$ is harmonic in the set $\{U>0\}\setminus (\overline{K}\times\{0\})$ and in the set $\R^{d+1}\setminus(\R^d\times\{0\})$.
\end{enumerate}
\end{lemma}

\begin{proof}
For \eqref{item:d+1 1} note first that if $f:\R\to\R$ is Lipschitz continuous with $f(0)=0$, then $f(v)\in H^1(\R^{d+1})$ for $v\in H^1(\R^{d+1})$ (see \cite{MM1979}). If $m:=\textnormal{essinf }U<0$, we may take $\epsilon\in(0,-m)$ and consider the Lipschitz function $f:\R\to\R$, $f(t)=\max\{t,-\epsilon\}$. Then $\calL_d(\{\tr(f(U))>0\})=\calL_d(\{\tr (U)>0\})$ and, for $\epsilon$ small enough,
$$
\int_{\R^{d+1}} |\nabla f(U(x))|^2 \, dx=\int_{\R^{d+1}} (f'(U(x))^2|\nabla U(x)|^2 \, dx=\int_{\{U>-\epsilon\}} |\nabla U(x)|^2 \, dx<\int_{\R^{d+1}} |\nabla U(x)|^2 \, dx.
$$
But then $E_{\lambda}(f(U))<E_{\lambda}(U)$, a contradiction. Thus $U\geq 0$. Similarly, with $f(t)=\min\{t,1+\epsilon\}$ it follows that we must have $U\leq 1$. \\
Note that $U$ is in particular a local minimizer in the sense as studied in \cite{CRS2010}. To be precise, for every $W\in H^1(B)$, where $B$ is any ball in $\R^{d+1}$ with center $x\in \R^{d}\times\{0\}$ such that $\overline{B}\cap  (\overline{K} \times \{0\})=\emptyset$ and $U=W$ on $\partial B$ we have 
$$
J(U,B)\leq J(W,B),
$$
where
$$
J(u,B)=\int_B|\nabla u(x)|^2\ dx+ \frac{\pi}{4}\lambda^2 \calL_d(\{\tr u>0\}\cap B).
$$
 By \cite[Theorem 1.1]{CRS2010} it then follows that $U\in C^{1/2}(M)$ for any compact set $M\subset \R^{d+1}\setminus (\overline{K}\times\{0\})$, that is, \eqref{item:d+1 2} holds. \\
To see \eqref{item:d+1 3}, note that $\{U>0\}$ is open by \eqref{item:d+1 2} and thus we may pick $\psi\in C^{\infty}_c(\{U>0\}\setminus (\overline{K}\times\{0\}))$ and consider $U+t\psi$ for $t\in \R$. We then have
\begin{align*}
 0&= \lim_{t\to0} \frac{E_{\lambda}(U+t\psi)-E_{\lambda}(U)}{t}\\
&=2 \int_{\R^{d+1}} \nabla U(x)\cdot \nabla \psi(x)\ dx + \frac{\pi}{4}\lambda^2 \lim_{t\to0}\frac{\calL_d(\{\tr(U+t\psi)>0\})-\calL_d(\{\tr(U)>0\})}{t}\\
&=2 \int_{\R^{d+1}} \nabla U(x)\cdot \nabla \psi(x)\ dx,
\end{align*}
so that $U$ is harmonic in $\{U>0\}\setminus (\overline{K}\times\{0\})$. Similarly, picking $\psi\in C^{\infty}_c(\R^{d+1}\setminus (\R^{d}\times\{0\}))$ we have
$$
0=\lim_{t\to0} \frac{E_{\lambda}(U+t\psi)-E_{\lambda}(U)}{t}=2 \int_{\R^{d+1}} \nabla U(x)\cdot \nabla \psi(x)\ dx
$$
so that $U$ is harmonic in $\R^{d+1}\setminus (\R^{d}\times\{0\})$. 
\end{proof}

In the following, let $\lambda>0$ be a fixed constant.

\begin{lemma}
\label{monotonicity}
Let $B\subset K\subset \R^d$ be open nonempty bounded sets. If $U$ is minimizer of $E_{\lambda}$ subject to the constraint $\tr U=1$ in $K$ and $V$ is a minimizer of $E_{\lambda}$ subject to the constraint $\tr V=1$ in $B$, then $U\geq V$.
In particular, there is at most one solution of Problem \ref{variational}
\end{lemma}
\begin{proof}
We follow part of the proof of \cite[Lemma 3.8]{AM1995}. By definition, $W:=\max\{U,V\}\in H^1(\R^{d+1})$ satisfies $\tr(W)=1$ in $K$. Moreover, the function $w:=\min\{U,V\}\in H^1(\R^{d+1})$ satisfies $\tr (w)=1$ in $B$. For $u\in H^1(\R^{d+1})$ we let
$$
R_u:=\int_{\R^{d+1}}|\nabla u(x)|^2\ dx.
$$
Since $V$ is a minimizer of $E_{\lambda}$ subject to the constraint $\tr(V)=1$ in $B$, we have
\begin{align*}
R_V+& \frac{\pi}{4}\lambda^2 \Big(\calL_d(\{ \tr (w)>0\})+\calL_d(\{ \tr (V)>0\}\setminus\{\tr(w)>0\})\Big)\\
&=R_V+ \frac{\pi}{4}\lambda^2 \calL_d(\{ \tr (V)>0\})=E_{\lambda}(V)\leq E_{\lambda}(w)=R_{w}+ \frac{\pi}{4}\lambda^2  (\calL_d(\{ \tr (w)>0\}),
\end{align*}
that is
$$
R_V+ \frac{\pi}{4}\lambda^2 \calL_d(\{ \tr (V)>0\}\setminus\{\tr w >0\})\leq R_w.
$$
Using that $R_W+R_w=R_V+R_U$ holds, we find
\begin{align*}
E_{\lambda}(W)&=R_W+ \frac{\pi}{4}\lambda^2 \calL_d(\{ \tr W>0\})\\
&=R_V+R_U-R_w+ \frac{\pi}{4}\lambda^2 \Big(\calL_d(\{ \tr (U)>0\})+\calL_d(\{ \tr (V)>0\} \setminus \{\tr(w)>0\})\Big)\leq E_{\lambda}(U).
\end{align*}
Hence, $E_{\lambda}(W)=E_{\lambda}(U)$, that is, $W$ is a minimizer of $E_{\lambda}$ subject to the constraint $\tr(W)=1$ in $K$. In particular, with a similar calculation we find
\begin{align*}
E_{\lambda}(w)&=R_w+ \frac{\pi}{4}\lambda^2 \calL_d(\{ \tr w>0\})=R_V+R_U-R_W+ \frac{\pi}{4}\lambda^2 \calL_d(\{ \tr w>0\})\\
&=R_V+ \frac{\pi}{4}\lambda^2 \calL_d(\{ \tr W>0\})-\frac{\pi}{4}\lambda^2\calL_d(\{\tr U>0\})+\frac{\pi}{4}\lambda^2 \calL_d(\{ \tr w>0\})\\
&=R_V+\frac{\pi}{4}\lambda^2\calL_d(\{ \tr (V)>0\} \setminus \{\tr(w)>0\})\Big)+\frac{\pi}{4}\lambda^2 \calL_d(\{ \tr w>0\})=E_{\lambda}(V).
\end{align*}
Hence, $w$ is a minimizer of $E_{\lambda}$ subject to the constraint $\tr(w)=1$ in $B$. It remains to show $W=U$ or equivalently $w=V$. Following \cite[Lemma 8.1]{ACF1984}, assume on the contrary that there is $x_0\in\{V>0\}\cap \{U>0\}$ with $V(x_0)=U(x_0)$ and the function $U-V$ changes sign in any neighborhood of $x_0$. Note that this implies that $w$ is not harmonic in any neighborhood of $x_0$, but this contradicts Lemma \ref{general properties}\eqref{item:d+1 3}, since $w$ is a minimizer of $E_\lambda$ subject to the constraint $w=1$ in $B$, $w(x_0)>0$, and $w$ is continuous in a small enough neighborhood of $x_0$.
\end{proof}

\begin{lemma}
\label{existence}
There exists a unique solution of Problem \ref{variational}.
\end{lemma}
\begin{proof}
The existence of a solution of Problem \ref{variational} follows from a standard argument (see \cite[Theorem 1.1]{ACF1984}, \cite{AC1981} or \cite[Proposition 3.2]{CRS2010}). Let $r_0=\inf\{r>0\;:\; K\subset B_r^d(0)\}$. Let $X:=\{ u \in H^{1}(\R^{d+1})\;:\; \tr(u) =1$ on $K\}$ and $X_r:=\{u \in H^{1}( B_r^{d+1}(0))\;:\; \tr(u)=1$ on $K\}$ for $r>r_0$. Clearly, $X$ and $X_r$ are nonempty and convex and there is $U_1\in X$ such that $E_{\lambda}(U_1)<\infty$. Let $(U_k)_k\subset X$ be a minimizing sequence. Then $(U_k)_k$ is bounded in $X_r$ for any $r>r_0$. But then there is $U\in X$ such that up to a subsequence
\begin{align}
U_k\rightharpoonup\ & U &&\text{ weakly in $X_r$ for any $r>r_0$,}\label{weak-d+1}\\
U_k\to\ & U&& \text{ almost everywhere in $\R^{d+1}$, and}\label{pointwise-d+1}\\
1_{\{ U_k>0\}} \stackrel{\ast}{\rightharpoonup}\ &  \gamma &&\text{in $L^{\infty}(\R^d)$,} \label{weakstar-d+1}
\end{align}
where $\gamma\in L^{\infty}( \R^d)$ satisfies $0\leq \gamma\leq 1$ and it is equal to $1$ almost everywhere in $\{(U)>0\}$. Note that \eqref{pointwise-d+1} follows by the compact embedding of $H^{1}( B_r^{d+1}(0))$ into $L^2( B_r^{d+1}(0))$ so that $U_k$ converges strongly in $L^2( B_r^{d+1}(0))$. \eqref{weakstar-d+1} follows from the fact that any bounded sequence in $L^{\infty}(\R^d)$ has a weak star convergent subsequence. To conclude, for any $r>r_0$ we have
\begin{align*}
&\int_{ B_r^{d+1}(0)}  |\nabla U(x)|^2 \, dx+ \frac{\pi}{4}\lambda^2\int_{ B_r^d(0)}\gamma(x)\ dx \\
&\leq \liminf_{k\to\infty} \int_{B_r^{d+1}(0)} |\nabla U_k(x)|^2 \, dx+\frac{\pi}{4}\lambda^2 \lim_{k\to\infty}\calL_d(\{\tr  U_k >0\})\\
&\leq \liminf_{k\to\infty}E_{\lambda}(U_k).
\end{align*}
Sending $r\to\infty$, it follows that $U$ is a minimizer of $E_{\lambda}|_{X}$, that is $U$ solves Problem \ref{variational}. The uniqueness now follows from Lemma \ref{monotonicity}.
\end{proof}

\begin{lemma}
\label{radial functions}
Let $g:(0,\infty)\to [0,\infty)$ be nonincreasing and assume there is $r_0>0$ such that 
$$
\lim_{t\to r_0^-}g(t)-\lim_{t\to r_0^+}g(t)=\epsilon>0.
$$
Then $u:\R^d\to\R$, $u(x)=g(|x|)$ is not in $H^{1/2}(\R^d)$. In particular, the harmonic extension $U(x)=\int_{\R^d}P(x-y)u(x)\ dy$, where $P$ is as in \eqref{extension}, is not in $H^1(\R^{d+1}_+)$.
\end{lemma}

The statement of this Lemma seems to be known, but since we could not find a good reference, we include an elementary proof in the appendix.

\begin{lemma}
\label{existence in ball}
Let $B$ be a ball in $\R^{d}$. Then for any $\lambda>0$ the minimizer $V$ of $E_{\lambda}$ subject to the constraint $\tr V =1$ in $B$, has up to translation a trace which is radially symmetric and decreasing in the radial direction. In particular, $\{\tr V>0\}$ is bounded.\\
Moreover, $\partial\{\tr V >0\}\cap \overline{B}=\emptyset$ and $\tr V $ is continuous in $\R^{d}$.
\end{lemma}
\begin{proof}
In the following, we may assume without restriction that $B=B_r(0)$ for some $r>0$. Note next that a unique minimizer $V$ of $E_{\lambda}$ exists with $\tr V =1$ in $B$ by Lemma \ref{existence}. Recall that the polarization of a function $U$ for a given hyperplane $H=\{x\cdot e>c\}$, $e\in \R^{d+1}$, $|e|=1$, $c\geq 0$ is given by 
$$
U_H(x)=\left\{\begin{aligned} &\max\{u(x),u(R_H(x))\} &&x\in H;\\
&\min\{u(x),u(R_H(x))\} && x\in \R^{d+1}\setminus H, \end{aligned}\right.
$$
where $R_H:\R^{d}\to\R^{d}$ denotes the reflection of a point at $\partial H$. The polarization satisfies
$$
\int_{\R^{d+1}} |\nabla u_H(x)|^2\ dx=\int_{\R^{d+1}} |\nabla u(x)|^2\ dx
$$
and thus $E_{\lambda}(U_H)\leq E_{\lambda}(U)$ if $e\in \R^{d}\times\{0\}$. In the following, denote by $V^{\ast}\in H^1(\R^{d+1})$ a function such that $\tr V^{\ast} $ is the radially symmetric rearrangement of $\tr V$, which is decreasing in the radial direction. Note that there exists a sequence of hyperplanes $H_1,\ldots$ such that (see \cite{VS2009})
$$
\lim_{n\to\infty} ((V_{H_1})\ldots)_{H_n}=V^{\ast}
$$
strongly in $L^2(\R^{d+1})$ and thus it follows that $E_{\lambda}(V^{\ast})\leq E_{\lambda}(V)$. Since by construction also $\tr V^{\ast}=1$ on $B$, by uniqueness we must have $V=V^{\ast}$. To see that $\partial\{\tr V >0\}\cap \overline{B}=\emptyset$ it is enough to note that $\tr V$ must be continuous by Lemma \ref{radial functions} and Lemma \ref{general properties}\eqref{item:d+1 3}.
\end{proof}

\begin{cor}
\label{boundedness}
Let $K\subset \R^d$ be an open nonempty bounded set and let $U$ be a minimizer of $E_{\lambda}$ subject to the constraint $\tr U=1$ on $K$. Then $\{\tr  U >0\}$ is bounded.
\end{cor}
\begin{proof}
Since $K$ is a bounded, there is a ball $B\subset \R^d$ with $K\subset B$. Letting $V$ be given by Lemma \ref{existence in ball} in $B$, Lemma \ref{general properties} and Lemma \ref{monotonicity} imply $0\leq U\leq V$. In particular, it follows that $\{\tr U >0\}\subset \{\tr V>0\}$ and the claim follows.
\end{proof}

\begin{cor}
\label{boundary and representation} 
Assume $K\subset \R^d$ is open nonempty and bounded with $C^2$ boundary. Let $U\in H^1(\R^{d+1})$ be a minimizer of $E_{\lambda}$ subject to the constraint $\tr U=1$ in $K$ given by Lemma \ref{existence}. Then the sets $\partial\{\tr U>0\}$ and $\overline{K}$ are disjoint. In particular $\tr  U$ is continuous in $\R^{d}$ and $U$ can be represented in $\R^{d+1}_+$ by the harmonic extension of $\tr U $ and hence $U>0$ in $\R^{d+1}_+$.
\end{cor}
\begin{proof}
For the first part we proceed as in \cite[Lemma 3.8]{AM1995}. Since $K$ is with $C^2$ boundary, we may fix a ball $B$ tangentially contained in $K$. Without loss of generality we may assume that $B$ has center $0$. Having a radially symmetric and continuous minimizer $V$ of $E_{\lambda}$ subject to the constraint $\tr V=1$ in $B$ by Lemma \ref{existence in ball}, Lemma \ref{monotonicity} implies $U\geq V$. The fact that $\{\tr V>0\}$ does not touch $\overline{B}$ and since $B$ is chosen arbitrary implies $\partial\{\tr U>0\}$ and $\overline{K}$ are disjoint. Moreover, it follows that $\tr U$ is continuous at $\partial K$ (and thus in $\R^d$) since $\tr V$ is continuous.

Lemma \ref{general properties} and the first part imply that $\tr U\in H^{1/2}(\R^{d})\cap C(\R^d)$ and $U$ is harmonic in $\R^{d+1}_+$. Since moreover $0\leq U\leq 1$ the function can be represented by the harmonic extension of its trace in $\R^{d+1}_+$, see \eqref{extension}. In particular, since $\tr U\geq0$, it follows that $U>0$ in $\R^{d+1}_+$.
\end{proof}

\begin{lemma}
\label{nonlocal-framework}
Let $K\subset \R^d$ be an open nonempty bounded set with $C^2$ boundary. Let $U\in H^1(\R^{d+1})$ be a minimizer of $E_{\lambda}$ subject to the constraint $\tr U=1$ on $K$, and denote $\varphi:=\tr U\in H^{1/2}(\R^d)$. Then 
\begin{enumerate}[(i)]
\item\label{item:d 1} $\varphi$ satisfies $0 \le \varphi \le 1$ and
\item\label{item:d 2} $\varphi$ is $1/2$-H{\"o}lder continuous on any compact subset of $\R^{d}\setminus \overline{K}$.
\item\label{item:d 3} $\varphi$ is a minimizer among all functions in $H^{1/2}(\R^d)$ that are equal to $1$ almost everywhere on $K$ of the functional  $e_{\lambda}$.
In particular, $\varphi$ is $1$-harmonic in $\{\varphi>0\}\setminus \overline{K}$.
\end{enumerate}
\end{lemma}
\begin{proof}
The statements of \eqref{item:d 1} and \eqref{item:d 2} follow immediately from Lemma \ref{general properties}. Using the properties of the extension we find with the properties of $U$ given by Lemma \ref{general properties}:
\begin{align*}
\inf_{\substack{V\in H^1(\R^{d+1}) \\ \tr(V)=1\, \text{in $K$}}} E_{\lambda}(V)=E_{\lambda}(U)=2\frac{\calA_d}{2}[\varphi]_1+ \frac{\pi}{4}\lambda^2 \calL_d(\{\varphi>0\})=e_{\lambda}(\varphi).
\end{align*}
On the other hand, taking any function $v\in H^{1/2}(\R^d)$ with $v=1$ on $K$ and considering by $V$ its harmonic extension in $\R^{d+1}_+$, we may extend $V$ symmetrically to a function on $\R^{d+1}$. It follows that $V\in H^1(\R^{d+1})$ satisfies $\tr V=1$ on $K$ so that it follows 
$$
 e_{\lambda}(v)=E_{\lambda}(V)\geq E_{\lambda}(U)=e_{\lambda}(\varphi).
$$
Hence, $\varphi$ is a minimizer of  $e_{\lambda}$ as claimed. The $1$-harmonicity of $\varphi$ in $\{\varphi>0\}\setminus \overline{K}$ follows analogously to the proof of Lemma \ref{general properties}\eqref{item:d+1 3} by taking $\psi\in C^{\infty}_c(\{\varphi>0\}\setminus \overline{K})$ and noting that
$$
0=\lim_{t\to0}\frac{e_{\lambda}(\varphi+t\psi)-e_{\lambda}(\varphi)}{t}=2 \calA_d\iint_{\R^d\times\R^d}\frac{(\varphi(x)-\varphi(y))(\psi(x)-\psi(y))}{|x-y|^{d+1}} \, dx \, dy.
$$
\end{proof}

\begin{rem}\label{nonlocal-remarks}
\begin{enumerate}[(i)]
\item\label{item:nonlocal-existence} Note that similarly to Lemma \ref{general properties} and Lemma \ref{existence} it can be shown that there exists a minimizer $v\in H^{1/2}(\R^d)$ of $e_{\lambda}$ with $v=1$ in $K$ (see \cite[Proposition 3.2]{CRS2010}) and this minimizer satisfies $0\leq v\leq 1$ with a similar argument.
\item\label{item:nonlocal-equivalence} Furthermore, note that also the converse of Lemma \ref{nonlocal-framework} is true in the following sense: If $v\in H^{1/2}(\R^d)$ is any minimizer of $E_{\lambda}$ among all functions satisfying $v=1$ in $K$, and $V$ denotes the harmonic extension to $\R^{d+1}_+$. Then, by denoting again with $V$ the symmetric extended function $V\in H^{1}(\R^{d+1})$ we have with the first remark
$$
E_{\lambda}(V)=e_{\lambda}(v)=\inf_{\substack{w\in H^{1/2}(\R^d)\\ w=1\,\text{in $K$}}}e_{\lambda}(w)=\inf_{\substack{W\in H^1(\R^{d+1})\\ \tr(W)=1\,\text{in $K$}}} E_{\lambda}(W),
$$  
where we have used that due to Lemma \ref{boundary and representation} any minimizer of $E_{\lambda}$ with trace equal to $1$ in $K$ is bounded and harmonic in $\R^{d+1}_+$ and can be represented by the harmonic extension of its trace. In particular, it follows that such a minimizer $v$ of $E_{\lambda}$ is locally $1/2$-H\"older continuous in $\R^d\setminus\overline{K}$ and $1$-harmonic in $\{v>0\}\setminus\overline{K}$. Here, the set $K$ does not need to have a $C^2$ boundary.
\end{enumerate}
\end{rem}

\begin{lemma}
\label{starshaped}
 If an open nonempty bounded set $K\subset \R^d$ with $C^2$ boundary is starshaped with center $B_{\delta}^d(0)$ for some $\delta > 0$ then $\{U > \eps\}$ is starshaped with center $B_{\delta}^{d+1}(0)$ for all $0 \le \eps < 1$, where $U\in H^1(\R^{d+1})$ is a minimizer of $E_{\lambda}$ subject to the constraint $\tr U=1$ on $K$.
\end{lemma}
\begin{proof}
For any $r \ge 1$ and $V \in H^1(\R^{d + 1})$ put $\lambda_r = \lambda/\sqrt{r}$,
$$
J(r,V) = J_{\lambda_r}(V)=
\int_{\R^{d+1}} |\nabla V(x)|^2 \, dx + \frac{\pi}{4}\lambda_r^2 \calL_d (\{ \tr V  > 0\}),
$$
$U_r(x) = U(x/r)$, $U_r^+(x) = \max\{U(x),U_r(x)\}$, $U_r^-(x) = \min\{U(x),U_r(x)\}$, $rK = \{rx: \, x \in K\}$.
We have $\tr U_r \ge \tr U$ on $rK $ so $\tr U_r^+= \tr U_r$ and $\tr U_r^- = \tr U$ on $rK$. We have
\begin{equation}
\label{Jm}
J(1,U) \le J(1,U_r^-)
\end{equation}
because $U$ minimizes $J(1,v)$ in $H^1(\R^{d + 1})$ subject to the constraint that $\tr v = \tr U $ on $K$. 

Similarly
$$
J(r,U_r) \le J(r,U_r^+)
$$
because $U_r$ minimizes $J(r,v)$ in $H^1(\R^{d + 1})$ subject to the constraint that $\tr v  = \tr U_r $ on $rK$. 

For $r > 1$ we also have
\begin{eqnarray*}
0 &\ge& J(r,U_r) - J(r,U_r^+)\\
&=& J(1,U_r) - J(1,U_r^+) + \left(\frac{\pi}{4}\frac{\lambda^2}{r} - \frac{\pi}{4}\lambda^2\right) (\calL_d(\{\tr U_r  > 0\}) - \calL_d(\{\tr  U_r^+ > 0\}))\\
&\ge& J(1,U_r) - J(1,U_r^+).
\end{eqnarray*}
We conclude that 
\begin{equation}
\label{Jp}
J(1,U_r) \le J(1,U_r^+).
\end{equation}

However,
\begin{equation}
\label{Jsum}
J(1,U) + J(1,U_r) = J(1,U_r^-) + J(1,U_r^+).
\end{equation}
By combining (\ref{Jm}), (\ref{Jp}) and (\ref{Jsum}) we get
$$
J(1,U) = J(1,U_r^-), \quad \quad J(1,U_r) = J(1,U_r^+).
$$
It follows that by Lemma \ref{monotonicity}
$$
U = U_r^- \le U_r^+ = U_r \quad \text{on} \quad \R^{d + 1}
$$
so that $\{U > \eps\}$ is starlike with respect to the origin. The proof can be repeated with the origin replaced by any $x_0 \in B_{\delta}^d(0)$, to show that $\{U > \eps\}$ is starlike with respect to all points in $B_{\delta}^{d+1}(0)$.
\end{proof}

\begin{lemma}
\label{solution to bernoulli}
Let $U$ be the uniquely determined solution of Problem \ref{variational}, where $K\subset \R^d$ is an open nonempty bounded set, has $C^2$ boundary and is starshaped with center $B_{\rho}(x_0)$ for some $\rho>0$ and $x_0\in K$. Denote $u:=\tr U$. Then $\partial \{u>0\}$ is of class $C^{\infty}$ and $u$ is a solution of Problem \ref{ebp}.
\end{lemma}
\begin{proof} (\ref{bernoulli-problem basis}) holds by Lemma \ref{nonlocal-framework}. Next, note that by the properties of $K$ we have $u\in C(\R^d)\cap H^{1/2}(\R^d)$ and $\{u>\epsilon\}$ is starshaped with center $B_{\rho}(x_0)$ for any $\epsilon\geq0$. Thus $\partial\{u>0\}$ is of class $C^{0,1}$ and (\ref{bernoulli-problem basis}) in particular also hold in a viscosity sense. Next, note that by denoting for $V\in H^1(\R^{d+1})$ the dilation $V_{c}(x):=V(c x)$ for $c>0$, we have
\begin{align*}
&\int_{\R^{d+1}}|\nabla V_c(x)|^2 \ dx+ \calL_d(\{x\in \R^d\;:\; \tr(V_c)(x)>0\})\\
&=\int_{\R^{d+1}}|\nabla V(c x)|^2 c^2\ dx+ \calL_d(\{x\in \R^d\;:\; \tr(V)(c x)>0\})\\
&=c^{1-d}\int_{\R^{d+1}}|\nabla V(z)|^2\ dz+ c^{-d} \calL_d(\{z\in \R^d\;:\; \tr(V)(z)>0\})\\
&=c^{1-d} (\int_{\R^{d+1}}|\nabla V(z)|^2\ dz+ c^{-1} \calL_d(\{z\in \R^d\;:\; \tr(V)(z)>0\})).
\end{align*}
Hence, if $U$ is a minimizer of $E_{\lambda}$ subject to the constraint $\tr U=1$ on $K$, then the dilation $U_{\frac{4}{\pi\lambda^2}}$ is a local minimizer as studied in \cite{CRS2010} subject to the constraint that the trace is $1$ on $\frac{\pi \lambda^2}{4} K$. Thus we have for a.e. $\theta\in \partial\{u>0\}$ by \cite[Theorem 1.4]{CRS2010}, see also \cite[Proposition 2.1]{FR2020} for the constant, noting that $\nu(\frac{\pi \lambda^2}{4} \theta)=\nu(\theta)$,
$$
\lim_{t\to 0^+} \frac{u(\theta+t\nu(\theta))}{\sqrt{t}}=\lim_{t\to 0^+} \frac{\tr U(\theta+t\nu(\theta))}{ \sqrt{t}}=\lim_{t\to 0^+} \frac{\sqrt{\frac{\pi \lambda^2}{4}}\tr U_{\frac{4}{\pi\lambda^2}}(\frac{\pi \lambda^2}{4}\theta+ \frac{\pi \lambda^2}{4}t\nu(\frac{\pi \lambda^2}{4}\theta))}{\sqrt{\frac{\pi \lambda^2}{4}t}}=\frac{\sqrt{\frac{\pi \lambda^2}{4}}}{\Gamma(3/2)}=\lambda
$$
as claimed. Then, by \cite[Theorem 1.2]{DSS2015} and \cite[Theorem 1.1]{DSS2015b} it follows that $\partial\{u>0\}$ is of class $C^\infty$ and thus (\ref{normal_exterior}) holds for all $\theta\in \partial\{\varphi>0\}$ as claimed.
\end{proof}

\begin{proof}[Proof of Theorem \ref{bernoulli-exterior}]
The existence and uniqueness is given by Lemma \ref{existence}. Properties (a)--(c) follow from Lemma \ref{general properties}. (d) follows from Corollary \ref{boundedness}. By standard arguments and \cite[Theorem 1.4]{CRS2010} (see also \cite[Proposition 2.1]{FR2020} for the constant) we obtain (e). If $K$ is in addition with $C^2$ boundary, Corollary \ref{boundary and representation} yields the continuity of the minimizer, and if $K$ is starshaped with center $B_r(x_0)$ for some fixed $x_0\in K$, $r>0$, then the last set of assertions follows from Lemma \ref{starshaped} and Lemma \ref{solution to bernoulli}.
\end{proof}

We close this section with a remark concerning the solution of Problem \ref{ebp}. 

\begin{prop}
\label{uniqueness of ebp}
Let $K\subset \R^d$ be an open nonepmpty and bounded set and $\lambda>0$. Assume there is $u\in C(\R^d)$ such that $\overline{K}\subset \{u>0\}$, $\{u>0\}$ is bounded and of class $C^2$, 
\begin{equation}
\label{bernoulli-problem basis 2}
\left\{
\begin{aligned}
(-\Delta)^{1/2}u&=0 &&\text{in $\{u>0\}\setminus \overline{K}$,}\\
u&=1 &&\text{in $\overline{K}$,}\\
\end{aligned}
\right.
\end{equation}
and $\lim\limits_{t\to 0^+} \frac{u(\theta+t\nu(\theta))}{\sqrt{t}} =\lambda$ for all $\theta\in \partial \{u>0\}$. Then $u$ is uniquely determined. In particular, if $K$ is in addition of class $C^2$ and starshaped with respect to $B_r(x_0)$ for some $x_0\in\R^d$ and $r>0$, then there is one and only one solution of Problem \ref{ebp} given by Theorem \ref{bernoulli-exterior}.
\end{prop}
\begin{proof}
In the following, we assume without loss of generality that $0\in K$. Let $u_1,u_2$ be two solutions. Note that the strong maximum principle implies $0<u_i<1$ in $D_i\setminus \overline{K}$, $i=1,2$, where we set $D_i:=\{u_i>0\}$. Note that $w:=\min\{u_1,u_2\}$ is $1$-subharmonic in $D_w:=\{w>0\}\setminus K$, that is $(-\Delta)^{1/2}w(x) \ge 0$ for $x\in D_w$. If $D_i\neq \{w>0\}$, then the strong maximum principle implies $w<u_i$ in $\{w>0\}\setminus \overline{K}$ and thus we must have either $\{w>0\}=D_1$ or $\{w>0\}=D_2$. Without loss of generality, we may assume $\{w>0\}=D_2$, but then we have $w=u_2$. We thus have $u_1>u_2$ in $D_2\setminus \overline{K}$. Since $D_1$ is bounded, we can find $\epsilon \in(0,1)$ such that $\epsilon D_1\subset D_2$ and there is $\theta\in \partial (\epsilon D_1)\cap \partial D_2$. The function $v(x)=u_1(x/\epsilon)$ satisfies 
$$
(-\Delta)^{1/2} v=0\quad\text{in $\epsilon (D_1\setminus \overline{K})$,}\quad v=0\quad\text{in $\R^{d}\setminus \epsilon D_1$, and}\quad v=1\quad\text{in $\epsilon  K$.}
$$
Thus in particular, $v<u_2$ in $\epsilon  D_1\setminus \overline{K}$ by the strong maximum principle. In particular, we have at $\theta$
\begin{align*}
\lambda &=\lim\limits_{t\to 0^+} \frac{u_2(\theta+t\nu(\theta))}{\sqrt{t}}\geq \lim\limits_{t\to 0^+}\frac{v(\theta+t\nu(\theta))}{\sqrt{t}}=\lim\limits_{t\to 0^+} \frac{u_1(\frac{1}{\epsilon}\theta+ \frac{t}{\epsilon} \nu(\theta))}{\sqrt{t}}\\
&=\lim\limits_{t\to 0^+} \frac{u_1( \frac{1}{\epsilon}\theta+ t\nu(\theta))}{\sqrt{\epsilon}\sqrt{t}}=\frac{\lambda}{\sqrt{\epsilon}},
\end{align*}
where we have used that $\frac{1}{\epsilon}\theta\in \partial D_1$ and $\nu(\theta)=\nu(\frac{1}{\epsilon}\theta)$. Since $\epsilon\in(0,1)$, this is clearly a contradiction and thus we must have $D_1=D_2$, but then $u_1=u_2$ as claimed.
\end{proof}

 \section{The interior Bernoulli problem for the half Laplacian}

 By similar arguments as in the proof of Lemma \ref{general properties} by replacing $U$ by $1-V$ we obtain the following result.
\begin{lemma}
\label{general properties interior}
Let $D\subset \R^d$ be an open nonempty bounded set, $\lambda>0$, and let $V\in H^1(\R^{d+1})$ be a minimizer of $I_{\lambda,D}$ subject to the constraint $\tr V=0$ on $\R^{d}\setminus D$. Then
\begin{enumerate}[(i)]
\item\label{item:d+1 1b} $0 \le V \le 1$,
\item\label{item:d+1 2b} $V$ is $1/2$-H{\"o}lder continuous on any compact subset of $\R^{d+1}\setminus ((\R^d\setminus D)\times\{0\})$,
and
\item\label{item:d+1 3b} $V$ is harmonic in the set $\{V<1\}\setminus ((\R^d\setminus D)\times\{0\})$ and in the set $\R^{d+1}\setminus(\R^d\times\{0\})$.
\end{enumerate}
\end{lemma}

We follow closely \cite{DK2010} to show that there is a nontrivial minimizer of Problem \ref{variational interior} for $\lambda$ large enough. To find a nontrivial minimizer, it will be useful to find a function $V\in H^1(\R^{d+1})\setminus\{0\}$ with $\tr V=0$ on $\R^d\setminus D$ such that
\begin{equation}
\label{small i lambda}
I_{\lambda,D}(V)<I_{\lambda,D}(0)=\frac{\pi}{4}\lambda^2\calL_d(D).
\end{equation}

\begin{prop}
\label{existence interior}
Let $D\subset \R^d$ be an open nonempty bounded set. Then $\Lambda(D)<\infty$. Moreover, for any $\lambda>\Lambda(D)$ it follows that there exists $V$ with $\tr V=0$ on $\R^d\setminus D$ such that \eqref{small i lambda} is satisfied and for any $\lambda\leq \Lambda(D)$ we have 
\begin{equation}
\label{lambda d characterization}
\min_{\substack{V\in H^1(\R^{d+1})\\ \tr V=0\ \text{on $\R^{d}\setminus D$}}}  I_{\lambda,D}(V)=I_{\lambda,D}(0)=\frac{\pi}{4}\lambda^2\calL_d(D).
\end{equation}
\end{prop}
\begin{proof}
We follow closely the steps described in \cite[Proposition 2.1]{DK2010}.
\begin{itemize}
\item[(C1)] If \eqref{small i lambda} holds, then there exists a nontrivial minimizer $V$ of $I_{\lambda,D}$ subject to the constraint $\tr V=0$ on $\R^{d}\setminus D$.
\end{itemize}
Indeed, the proof follows analogously to the existence part of the proof of Lemma \ref{existence} by choosing a minimizing sequence. The obtained minimizer then must be nontrivial by the assumption \eqref{small i lambda} and thus (C1) holds.
\begin{itemize}
\item[(C2)] For $\lambda$ large enough it follows that \eqref{small i lambda} holds.
\end{itemize}
To see this, let $\Phi\in C^{\infty}(\R^{d+1})$ with $\supp \tr \Phi\subset D$ and {tk $\calL_d(\{\tr \Phi>1\})>0$.} Then
\begin{align*}
I_{\lambda,D}(\Phi)-I_{\lambda,D}(0)&=\int_{\R^d}|\nabla \Phi(x)|^2\ dx+ \frac{\pi}{4}\lambda^2(\calL_d(\{\tr \Phi<1\}\cap D)-\calL_d(D))\\
&\leq \int_{\R^d}|\nabla \Phi(x)|^2\ dx-\frac{\pi}{4}\lambda^2 \calL_d(\tr \Phi\geq 1\}).
\end{align*}
Sending $\lambda\to \infty$, it follows that (C2) holds.
\begin{itemize}
\item[(C3)] The set of all $\lambda>0$ such that $I_{\lambda,D}$ has a nontrivial minimizer $V$ subject to the constraint $\tr V=0$ on $\R^d\setminus D$ is connected.
\end{itemize}
For this, let $\mu>0$ and assume that there is a nontrivial minimizer $V$ of $I_{\mu,D}$ subject to the constraint $\tr V=0$ on $\R^d\setminus D$. Note that we must have $I_{\mu}(V)\leq I_{\mu}(0)$ and thus in particular $\calL_d(\{\tr V<1\}\cap D)<\calL_d(D)$. But then for $\lambda>\mu$ we have
\begin{align*}
I_{\lambda,D}(V)-I_{\lambda,D}(0)\leq \frac{\pi}{4}(\lambda^2-\mu^2)\Big(\calL_d(\{\tr V<1\}\cap D)-\calL_d(D)\Big)<0.
\end{align*}
Hence \eqref{small i lambda} is satisfied and by (C1) it follows that (C3) is satisfied. Note that the proof of (C3) in particular implies that if $\lambda>\Lambda(D)$, then there exists $V$ with $\tr V=0$ on $\R^{d}\setminus D$ such that \eqref{small i lambda} is satisfied.\\
To finish the proof of the proposition, first note that if $\lambda<\Lambda(D)$, then clearly \eqref{lambda d characterization} holds due to (C1). Next, let $\mu=\Lambda(D)$ and let $V$ be a minimizer of $I_{\mu,D}$ subject to the constraint $\tr V=0$ on $\R^d\setminus D$. Assume $I_{\mu,D}(V)<I_{\mu,D}(0)$. Then it follows for $\lambda<\mu$, $\lambda$ close to $\mu$, that we have
$$
I_{\lambda,D}(V)-I_{\lambda,D}(0)=I_{\mu,D}(V)-I_{\mu,D}(0)+ \frac{\pi}{4}(\lambda^2-\mu^2)\Big(\calL_d(\{\tr V<1\}\cap D)-\calL_d(D)\Big)<0.
$$
Thus (C1) gives the existence of a nontrivial minimizer subject to the usual constraint in contradiction to the definition of $\Lambda(D)$.
\end{proof}

\begin{cor}
\label{representation interior}
Let $D\subset \R^d$ be an open nonempty bounded set. If $\lambda>\Lambda(D)$, then there exists a solution $V$ to Problem \ref{variational interior}, which is nonnegative and bounded in $\R^{d+1}$. In particular, $V$ is given in $\R^{d+1}_+$ by the harmonic extension of $u:=\tr V\in H^{1/2}(\R^d)$ and thus $V>0$ in $\R^{d+1}_+$. Moreover, $u$ satisfies
\begin{enumerate}[(i)]
\item $0\leq u \leq 1$
\item  $u$ is $1/2$-H\"older continuous on any compact subset of $D$.
\item $u$ is the minimizer of the functional $i_{\lambda,D}$ among all the functions in $H^{1/2}(\R^d)$ which vanish in $\R^{d}\setminus D$, and, in particular, $u$ is $1$-harmonic in $\{u<1\}\cap D$.
\end{enumerate}
\end{cor}
\begin{proof}
The existence follows by the same arguments as in the proof of Lemma \ref{existence}. The fact that $V$ is given by the harmonic extension of $\tr V$ follows by  standard arguments. The properties (i), (ii) follow from Lemma \ref{general properties interior}. The proof of (iii) is similar to the proof of (iii) in Lemma \ref{nonlocal-framework}.
\end{proof}

\begin{lemma}
\label{technical lemma}
Let $D\subset \R^d$ be open nonempty and bounded. Then there exists $C=C(d)>0$ and $p=p(d)>1$ such that 
\begin{equation}
\label{technical inequality}
\calL_d(\{\tr V\geq 1\}\cap D)\leq C\Big(\int_{\R^{d+1}}|\nabla V(x)|^2\ dx\Big)^{p}
\end{equation}
 for all nonnegative bounded $V\in H^1(\R^{d+1})$, $V$ even in $x_{d+1}$ with $\tr V=0$ on $\R^d\setminus D$ and such that $V$ is harmonic in $\R^{d+1}\setminus (\R^d\times\{0\})$. 
\end{lemma}
\begin{proof}
A function $V$ as stated can be represented by the harmonic extension of its trace. If $d>1$, then by the fractional Sobolev inequality (see e.g. \cite[Theorem 3.2.1]{BV2016}) we have
\begin{align*}
\calL_d(\{\tr V\geq 1\}\cap D)\leq \int_D|\tr V(x)|^{\frac{2d}{d-1}}\ dx\leq C[\tr V]_1^{\frac{d}{d-1}}=\frac{C}{\calA_d^{\frac{d}{d-1}}}\Big(\int_{\R^{d+1}}|\nabla V(x)|^2\ dx\Big)^{\frac{d}{d-1}},
\end{align*}
using the energy identity \eqref{energy identity}. If $d=1$, then $\tr V\in L^{q}(\R^d)$ for any $q<\infty$ and the claim follows similarly.
\end{proof}

\begin{prop}
\label{limit case}
Let $D\subset \R^d$ be open nonempty and bounded. Then for $\lambda=\Lambda(D)$ there exists a solution of Problem \ref{variational interior}.
\end{prop}
\begin{proof}
The proof follows with Lemma \ref{technical lemma} similar to the proof of \cite[Proposition 2.3]{DK2010}. By Definition of $\Lambda(D)$ and by Lemma \ref{existence interior}, there is a strictly decreasing sequence $(\lambda_n)_n\subset (\Lambda(D),\infty)$ with $\lim\limits_{n\to\infty}\lambda_n=\Lambda(D)$ such that for every $n\in \N$ there is a nontrivial minimizer $U_n\in H^1(\R^{d+1})$ subject to the constraint $\tr U_n=0$ on $\R^d\setminus D$ and we have
\begin{equation}
\label{starting property}
I_{\lambda_n,D}(U_n)=\int_{\R^{d+1}}|\nabla U_n(x)|^2\ dx+ \frac{\pi}{4}\lambda_n^2\calL_d(\{\tr U_n<1\}\cap D)\leq \frac{\pi}{4}\lambda_n^2 \calL_d(D)=I_{\lambda_n,D}(0).
\end{equation}
Let $r_0=\inf\{r>0\;:\; D\subset B_r^d(0)\}$. Let $X:=\{ u \in H^{1}(\R^{d+1})\;:\;  \tr(u) = 0$ on $D^c\}$ and $X_r:=\{u \in H^{1}(B_r^{d+1}(0))\;:\; \tr(u)=0$ on $D^c\}$ for $r>r_0$. Then, similarly as in the proof of Lemma \ref{existence}, there is $U \in H^1(\R^{d+1})$ such that up to a subsequence 
\begin{align*}
U_{n}\rightharpoonup\ & U &&\text{ weakly in $X_r$ for any $r>r_0$,}\\
U_n\to\ & U&& \text{ almost everywhere in $\R^{d+1}$, and}\\
 1_{\{\tr U_n<1\}\cap D}\stackrel{\ast}{\rightharpoonup}\ &  \gamma &&\text{in $L^{\infty}(\R^d)$,}
\end{align*}
where $\gamma\in L^{\infty}(\R^d)$ satisfies $0\leq \gamma\leq 1$ and it is equal to $1$ almost everywhere in $\{\tr U<1\}\cap D$.
Then, for any $r > r_0$ we have
\begin{align*}
\int_{ B_r^{d+1}(0)}|\nabla U(x)|^2\ dx+\frac{\pi}{4}\Lambda^2(D)\int_{B_r^d(0)}\gamma(x)\ dx
&\leq \liminf_{n\to\infty}  \,\int_{ B_r^{d+1}(0)}|\nabla U_n(x)|^2\ dx+\frac{\pi}{4}\lambda_n^2\calL_d(\{\tr U_n<1\}\cap D)\\
&\leq  \liminf_{n\to\infty}I_{\lambda_n,D}(U_n).
\end{align*}
Sending $r\to\infty$, we find for any $V\in H^1(\R^{d+1})$ with $\tr V=0$ on $\R^{d}\setminus D$
\begin{equation}
\label{minimizer in critical case}
I_{\Lambda(D),D}(U)\leq  \liminf_{n\to\infty}I_{\lambda_n,D}(U_n)\leq  \liminf_{n\to\infty}I_{\lambda_n,D}(V)= I_{\Lambda(D),D}(V).
\end{equation}
Hence $U\in H^1(\R^{d+1})$ is a minimizer of $I_{\Lambda(D),D}$ subject to the constraint $\tr U=0$ on $\R^{d}\setminus D$. To show that indeed $U\not\equiv 0$, assume on the contrary that $U\equiv 0$. Note that the map $T:H^1(\R^{d+1})\to L^p(D)$ is compact for any $p<\frac{2d}{d-1}$ (with $p<\infty$ if $d=1$). Indeed, this follows by the continous embedding $H^{1}(\R^{d+1})\to H^{1/2}(\R^{d})$ and the compact embedding of $H^{1/2}(\R^d)$ into $L^p(\R^d)$. Hence, we have $U_n\to 0$ in $L^q(D)$ for any $q<\frac{2d}{d-1}$ ($q<\infty$ for $d=1$). In particular,
$$
\calL_d(\{U_n\geq 1\}\cap D) \leq \int_{D} |\tr U_n(x)|^q\ dx\to 0\quad\text{for $n\to\infty$.}
$$
Note that due to \eqref{minimizer in critical case}, with $V=0$, we find that 
\begin{equation}
\label{contradiction part 1}
\int_{\R^{d+1}}|\nabla U_n(x)|^2\ dx\to 0\quad\text{for $n\to\infty$.}
\end{equation}
By Lemma \ref{technical lemma}, using that by Lemma \ref{general properties interior} and Corollary \ref{representation interior} the functions $U_n$ satisfy the assumptions of this Lemma, we then find some $C>0$ and $p>1$ such that
\begin{align*}
I_{\lambda_n,D}(U_n)&=I_{\lambda_n,D}(0)+\int_{\R^{d+1}}|\nabla U_n(x)|^2\ dx-\frac{\pi}{4}\lambda_n^2\calL_d(\{\tr U_n\geq 1\}\cap D)\\
&\geq I_{\lambda_n,D}(0)+\int_{\R^{d+1}}|\nabla U_n(x)|^2\ dx-C\frac{\pi}{4}\lambda_n^2\Big(\int_{\R^{d+1}}|\nabla U_n(x)|^2\ dx\Big)^{p}\\
&=I_{\lambda_n,D}(0)+\int_{\R^{d+1}}|\nabla U_n(x)|^2\ dx\Bigg(1-C\frac{\pi}{4}\lambda_n^2\Big(\int_{\R^{d+1}}|\nabla U_n(x)|^2\ dx\Big)^{p-1}\Bigg).
\end{align*}
By \eqref{contradiction part 1}, $p>1$, and since $\lambda_n\to \Lambda(D)$ for $n\to\infty$, it follows that there is $m\in \N$ such that
$$
1-C\frac{\pi}{4}\lambda_m^2\Big(\int_{\R^{d+1}}|\nabla U_m(x)|^2\ dx\Big)^{p-1}>0.
$$
But then $I_{\lambda_m,D}(U_m)>I_{\lambda_m,D}(0)$ in contradiction to \eqref{starting property}. Hence we must have $U\not\equiv 0$ as claimed.
\end{proof}

\begin{cor}
\label{extension limit case}
The statements of Corollary \ref{representation interior} extends to the case $\lambda=\Lambda(D)$.
\end{cor}
\begin{proof}
This follows directly from Proposition \ref{limit case} applied in the proofs for the existence.
\end{proof}

\begin{proof}[Proof of Theorem \ref{bernoulli-interior}]
This follows from Lemma \ref{general properties}, Corollary \ref{representation interior}, Proposition \ref{limit case}, Corollary \ref{extension limit case} and \cite[Theorem 1.4]{CRS2010} (see also \cite[Proposition 2.1]{FR2020} for the constant).
\end{proof}

\begin{proof}[Proof of Proposition \ref{monotonicity Lambda}]
Let $\lambda>\Lambda(D_1)$. Then there exists a nontrivial minimizer $V$ of $I_{\lambda,D_1}$ subject to the constraint $\tr V=0$ on $\R^{d}\setminus D_1$ and we have 
\begin{align*}
I_{\lambda,D_2}(V)&=I_{\lambda,D_1}(V)+\frac{\pi}{4}\lambda^2\calL_d(D_2\setminus D_1)\\
&< I_{\lambda,D_1}(0)+\frac{\pi}{4}\lambda^2\calL_d(D_2\setminus D_1)=\frac{\pi}{4}\lambda^2 \calL_d(D_1)+\frac{\pi}{4}\lambda^2\calL_d(D_2\setminus D_1)=\frac{\pi}{4}\lambda^2 \calL_d(D_2)=I_{\lambda,D_2}(0).
\end{align*}
The strict inequality is due to Proposition \ref{existence interior}. Hence, \eqref{small i lambda} holds and Proposition \ref{existence interior} --- in particular Claim (C1) in its proof --- implies $\lambda\geq\Lambda(D_2)$. Since $\lambda>\Lambda(D_1)$ was chosen arbitrarily, the first  claim of Proposition \ref{monotonicity Lambda} follows.

Regarding the homogeneity, it is enough to observe that if $V(x)=U(x/s)$, then $$\{\text{tr\,}V<1\}=s\{\text{tr\,}U<1\}\,,$$ 
whence
\begin{align*}
I_{\lambda/\sqrt{s},sD}(V)&=\frac1{s^2}\int_{\R^{d+1}} \left|(\nabla U)(x/s)\right|^2\,dx+\frac\pi4\frac{\lambda^2}{s}\calL_d(s(\{\text{tr\,}U<1\} \cap D))\\
\\
&=s^{d-1}\int_{\R^{d+1}}\left|\nabla U(x)\right|^2\,dx+\frac\pi4\lambda^2 s^{d-1}\calL_d(\{\text{tr\,}U<1\} \cap D)=s^{d-1}I_{\lambda,D}(U)\,.
\end{align*}
Then there is a one-to-one correspondence between nontrivial solutions of Problem \ref{variational interior} for the couple $(D,\lambda)$ and nontrivial solutions of the same problem for the couple $(sD,s^{-1/2}\lambda)$.
\end{proof}

\begin{proof}[Proof of Proposition \ref{iso inequality}]
Let $\lambda\geq \Lambda(D)$ and let $U$ be a nontrivial solution of Problem \ref{variational interior}. Consider as explained in the proof of Lemma \ref{existence in ball} $U^{\ast}$, which satisfies that $\tr U^{\ast}$ is the radial symmetric rearrangement of $\tr U$ such that it is nonincreasing in the radial direction. And note that via the polarization inequality we have similarly 
$$
I_{\lambda, B}(U^{\ast})\leq I_{\lambda, D}(U),
$$
since the symmetric rearrangement of $D$ is given by $B$ and we have $\calL_d(\{\tr U^{\ast}<1\}\cap B)=\calL_d(\{\tr U<1\}\cap D)$. Since $I_{\lambda,B}(0)=I_{\lambda,D}(0)$ we thus find
\begin{equation}
\label{inequality iso}
I_{\lambda, B}(U^{\ast})\leq I_{\lambda,B}(0).
\end{equation}
If this inequality is strict, then the existence of a nontrivial minimizer $V$ of $I_{\lambda,B}$ subject to the constraint $\tr V=0$ on $\R^{d}\setminus B$ follows by Claim 1 in the proof of Proposition \ref{existence interior} and we may conclude $\Lambda(D)\geq \Lambda(B)$.\\
If we have equality in \eqref{inequality iso}, then either $U^{\ast}$ is a nontrivial minimizer of $I_{\lambda,B}$ subject to the constraint $\tr U^{\ast}=0$ on $\R^{d}\setminus B$ of $I_{\lambda,B}$, or we must have 
$$
\inf_{\substack{V\in H^1(\R^{d+1})\\
\tr V=0\, \text{on $\R^{d}\setminus B$}}}  I_{\lambda,B}(V)<I_{\lambda,B}(0)
$$  
and thus there exists also a nontrivial minimizer $V$ of $I_{\lambda,B}$ subject to the constraint $\tr V=0$ on $\R^{d}\setminus B$. Thus also in the equality situation we may conclude $\Lambda(D)\geq \Lambda(B)$.\\
It remains to show that $\Lambda(D)=\Lambda(B)$ only holds if $D=B$ up to a set of measure zero. For this, assume $D\subset \R^d$ is any open nonempty bounded set with $\Lambda:=\Lambda(D)=\Lambda(B)$. Then we find also with the above choice of $U$ and $U^{\ast}$ for $\lambda=\Lambda$ that
$$
\frac{\pi}{4} \Lambda^2 \calL_d(D)= \frac{\pi}{4} \Lambda^2 \calL_d(B)=I_{\Lambda,B}(0)\leq I_{\Lambda,B}(U^{\ast})\leq I_{\Lambda,D}(U)= I_{\Lambda,D}(0)= \frac{\pi}{4} \Lambda^2 \calL_d(D).
$$
Thus $U^{\ast}$ is a nontrivial minimizer of $I_{\Lambda,B}$ subject to the constraint $\tr U^{\ast}=0$ on $\R^{d}\setminus B$ and it is uniquely determined.  Moreover, we also have $I_{\Lambda,B}(U^{\ast})= I_{\Lambda,D}(U)$. By \cite[Theorem 1.1]{BZ1988} or \cite[Theorem 1.1.]{FV2003}, it follows that we must have $U=U^{\ast}$ almost everywhere, but then $D=B$ almost everywhere and the claim follows.
\end{proof}

\begin{proof}[Proof of Lemma \ref{estimate bernoulli}]
By scaling we have $\Lambda(B_r^d(0)) = \Lambda(B_1^d(0))/\sqrt{r}$, so we may assume that $r = 1$ and we write $B=B_1^d(0)$ for simplicity. Put $\lambda = \frac{2 \sqrt{d}}{\sqrt{\pi}} 2^{(d+3)/2}$. By Proposition \ref{existence interior} is is enough to show that there exists $V \in H^1(\R^{d+1})$ such that $\tr V = 0$ on $\R^d \setminus B$ satisfying $I_{\lambda, B}(V) < \frac{\pi}{4}\lambda^2 \calL_d(B)$. This is equivalent to 
\begin{equation}
\label{gradient_estimate}
\int_{\R^{d+1}} |\nabla V(x)|^2 \, dx  < \frac{\pi}{4}\lambda^2 \calL_d(\{ \tr V \ge  1 \} \cap B ).
\end{equation}
We begin with the case $d\neq 2$. Denote $x = (\tilde{x},x_{d+1})$, where $\tilde{x} = (x_1,\ldots,x_d) \in \R^{d}$, $x_{d+1} \in \R$. Put $\tilde{\nabla} = 
\left(\frac{\partial}{\partial x_1}, \ldots, \frac{\partial}{\partial x_d}\right)$. Let $f \equiv 0$ on $B^c$, $f \equiv 1$ on $B_{1/2}^d(0)$, and $f(\tilde{x}) = (2^{d - 2} - 1)^{-1} (|\tilde{x}|^{2 - d} - 1)$ for $\tilde{x} \in B  \setminus B_{1/2}^d(0)$. Note that 
\begin{align*}
\int_{B  \setminus B_{1/2}^d(0)} |\tilde{\nabla} f(\tilde{x})|^2 \, d\tilde{x} 
&= \frac{(d-2)^2}{(2^{d-2}-1)^2}\int_{B  \setminus B_{1/2}^d(0)} |\tilde{x}|^{2-2d} \, d\tilde{x}  
%= \frac{2(d-2)^2\pi^{\frac{d}{2}}}{(2^{d-2}-1)^2\Gamma(\frac{d}{2})} \int_{1/2}^{1}r^{1-d}\ dr
=\frac{ 2(d-2)^2 \pi^{\frac{d}{2}}(2^{d}-1)}{d (2^{d-2}-1)^2\Gamma(\frac{d}{2} )}.
\end{align*}
Put $V(x) = f(\tilde{x}) e^{-d |x_{d+1}|}$. We have
\begin{align*}
\int_{\R^{d+1}} |\nabla V(x)|^2 \, dx 
&= \int_{B  \times \R} \left|\frac{\partial V}{\partial x_{d+1}}(x)\right|^2 \, dx 
+ \int_{B  \setminus B_{1/2}^d(0)} |\tilde{\nabla} f(\tilde{x})|^2 \, d\tilde{x} \int_{\R} e^{-2 d |x_{d+1}|} \, dx_{d+1}\\
&< \frac{2\pi^{\frac{d}{2}} }{ d\Gamma(\frac{d}{2} )}  \int_{-\infty}^{\infty} d^2 e^{-2d x_{d+1}} \, d x_{d+1} +\frac{ 2(d-2)^2 \pi^{\frac{d}{2}}(2^{d}-1)}{d^2 (2^{d-2}-1)^2\Gamma(\frac{d}{2} )}<9\frac{\pi^{\frac{d}{2}} }{ \Gamma(\frac{d}{2} )}.
\end{align*}
On the other hand, $\calL_d(\{tr V = 1 \} \cap B ) = \calL_d(B_{1/2}^d(0)) = \frac{\pi^{d/2}}{2^d (d/2) \Gamma(d/2)}$. Hence \eqref{gradient_estimate} holds, which gives the upper bound for $d\neq 3$.\\
For $d=2$ a similar calculation can be done with $V(\tilde{x},x_{d+1})=f(\tilde{x})e^{-d|x_{d+1}|}$, where $f$ is similar as above but with $f(\tilde{x})= -\ln|\tilde{x}|/\ln(2)$ for   $1/2<|\tilde{x}|<1$ and we leave the details for the reader.\\
For the lower bound, recall that by Proposition \ref{limit case} and Corollary \ref{extension limit case}, it follows that there is a continuous nontrivial minimizer $V_B$ of $I_{\lambda,B}$ with $\tr V=0$ on $\R^d\setminus B$ as long as $\lambda\geq \Lambda(B)$. Moreover, it follows that $v=\tr V_B$ is radially symmetric and nonincreasing in the radial direction. In particular, there is $\rho_v\in(0,1)$ such that $\{x\in \R^d\;:\; v(x)=1\}$ is given by $K_v=B_{\rho_v}^d(0)$. By Proposition \ref{existence interior} and Corollary \ref{representation interior}, we have
\begin{align*}
\frac{\pi}{4}\Lambda^2(B)\calL_d(B)
%&=I_{\lambda,B}(V_B)=\int_{\R^{d+1}} |\nabla V_B(x)|^2 \, dx + \Lambda(B)\,\calL_d (\{\tr V_B < 1\}\cap B)\\
&=\calA_d[v]_1 + \frac{\pi}{4}\Lambda^2(B)\,\Big(\calL_d(B)- \calL_d(K_v)\Big),
\end{align*}
so that
\begin{align*}
\frac{\pi}{4}\Lambda^2(B)&\geq \frac{\calA_d}{ \calL_d(K_v)}[v]_1\geq \frac{2\calA_d}{ \calL_d(K_v)}\int_{K_v}\int_{\R^d\setminus B}|x-y|^{-d-1}\ dydx.
\end{align*}
Since for $x\in K_v\cap B_{1/2}^d(0)$ and $y\in \R^d\setminus B$ we have $|x-y|\leq \frac{3}{2}|y|$ it thus follows that
\begin{align*}
\frac{\pi}{4}\Lambda^2(B)&\geq \frac{2\calA_d}{ \calL_d(K_v)}\int_{K_v\cap B_{1/2}^d(0)}\int_{\R^d\setminus B}|x-y|^{-d-1}\ dydx\geq \frac{\calA_d}{\rho_v^{d}} \frac{2^{d+1}}{3^{d+1}} \min(\rho_v^{d},2^{-d}) \int_{\R^d\setminus B}|y|^{-d-1}\ dy\\
&\geq  \frac{4 \calA_d\pi^{\frac{d}{2}}}{3^{d+1} \Gamma(\frac{d}{2} )} \int_{1}^{\infty}r^{-2}\ dr=\frac{2 \Gamma(\frac{d+1}{2} ) }{3^{d+1}\sqrt{\pi} \Gamma(\frac{d}{2} )}> \frac{\sqrt{d}}{3^{d+2} }
\end{align*}
as claimed.
\end{proof}

\section{The interior Bernoulli problem for the spectral half Laplacian}\label{section spectral}

In the whole section we assume that $d \ge 2$. Fix $\lambda > 0$ and an open, nonempty, bounded convex set $D \subset \R^d$. Recall the definition of $\calF(D,\lambda)$ given in Definition \ref{classF} for certains sets $K\subset D$ and associated functions $v\in C(\overline{D}\times \R)$. Note that $v \in \calF(D,\lambda)$ in particular satisfies $v \equiv 0$ on $\partial{D} \times \R$. We extend $v$ on the whole $\R^{d+1}$ by putting $v \equiv 0$ on $(\overline{D} \times \R)^c$.

 Let  $\R_+^d = \{z \in \R^{d}: \, z_1 > 0\}$, $e_1^d = (1,0,\ldots,0) \in \R^d$.

\begin{lemma}
\label{lemmadist}
If $K\in\calF(D,\lambda)$, then $\dist(K,\partial D)\geq\frac{1}{\lambda^2}$.
\end{lemma}
\begin{proof}
Let $w\in\partial D$ such that $\dist(K,\partial D)=\dist(w,K)$. Since $v_K(w,0)=0$, by the boundary conditions in \eqref{problemv}, and $v_K$ is continuous up to $\partial D$, by \eqref{overdetermined} we have
$$
\frac{1}{\text{dist}(\partial D,K)^{1/2}}=\lim_{y\to w}\frac{|v_{K}(y,0)-1|}{\delta_K^{1/2}(y)}\le\sup_{y\in D\setminus K} \frac{|v_{K}(y,0)-1|}{\delta_K^{1/2}(y)} \le \lambda\,,
$$
which gives the desired estimate.
\end{proof}

\begin{lemma}
\label{lemmasupv}
If $\calG\subseteq\calF(D,\lambda)$ and $K=\bigcup_{F\in\calG}F$, then $K^*=\overline{\text{conv}(K)}\in\calF(D,\lambda)$.
\end{lemma}
\begin{proof}
Notice first that $K\subseteq D$ by the convexity of $D$; indeed, $\dist(K,\partial D) \ge \frac{c}{\lambda^2}$ by the convexity of $D$ and Lemma \ref{lemmadist}.

For any $F\in \calG $ let $v_F$ be the solution of \eqref{problemv} for $F$, and set $$v(x)=\sup_{F\in \calG} v_F(x).$$ 
Then $\{x\,:\,v(x)=1\}=K$. Moreover, as it is well known that the supremum of a family of harmonic functions is subharmonic, we have that $v$ is subharmonic in $(D\times\R)\setminus(K\times\{0\})$.  Finally, we notice that 
$\delta_K(y)=\sup_{F\in\calG}\delta_F(y)$, then for every  $y \in D \setminus K$ we have
\begin{equation*}
\frac{|v(y,0) - 1|}{\delta_K(y)^{1/2}}=\frac{1-v(y,0)}{\delta_K(y)^{1/2}}=\sup_{F\in \calG} \frac{1-v(y,0)}{\delta_F(y)^{1/2}} \le \sup_{F\in \calG} \frac{1-v_F(y,0)}{\delta_F(y)^{1/2}}\,,
\end{equation*}
whence $v$ clearly satisfies \eqref{overdetermined} for $F\in\calF(D,\lambda)$ for every $F\in\calG$. Notice that we have by definition $0\leq v\leq 1$. By standard arguments $v \in C(\R^{d+1})$. Let $v_{K}$ denote the solution of \eqref{problemv} associated to $K$. Then since $v$ is subharmonic we have $v_{K}\geq v$ in $D\times \R$ and in particular $v\to 0$ for $|x|\to \infty$ since already $v_K$ has this property as mentioned in the introduction.

Now let $v^*$ be the quasi-concave envelope of $v$. We understand the quasi-concave envelope of a function as in \cite{CS2003}, that is the function $v^*$ whose superlevel sets are the convex hulls of the corresponding superlevel sets of $v$. We recall that $v^*$ is explicitly defined as follows
$$
v^*(x)= \sup\left\{\min(v(x_1),\dots,v(x_{d+2}))\,:\,x_1,\dots,x_{d+2}\in \R^{d+1},\,\mu\in\Gamma_{d+2},\,\sum_{i=1}^{d+2}\mu_ix_i=x\right\}\,,
$$
where $\Gamma_k=\{\mu=(\mu_1,\dots,\mu_k)\,:\,\mu_i\geq0,\,\sum_{i=1}^k\mu_i=1\}$ for $k\geq 2$.
Notice that $v^*\geq v$ and $\{v^*=1\}=K^*$ by the very definition, and $v^*\in C(\overline{D}\times\R)$ by \cite[Lemma 2.2]{CS2003} using that we have $v\to 0$ for $|x|\to \infty$ and thus also $v^*\to0$ for $|x|\to \infty$.
Furthermore, Theorem 3.2 of \cite{CS2003} implies that 
\begin{equation}\label{subharmonic-convex}
\Delta v^*\geq 0\,\,\,\text{ in the viscosity sense in } (D\times\R) \setminus(K^*\times\{0\})\,.
\end{equation}
To see that $v^*$ indeed satisfies \eqref{overdetermined}, let $y\in D\setminus K^*$ (otherwise there is nothing to show) and fix the unique $x_y\in \partial K^*$ such that $\dist(y,K^*)=|y- x_y|$, which is possible, since $K^*$ is convex. If $ x_y\in \partial K$, then we have immediately 
$$
  \frac{|v^*(y,0) - 1|}{\delta_{K^{\ast}}^{1/2}(y)}=  \frac{1 - v^*(y,0)}{|x_y-y|^{1/2}}\leq \frac{1 - v(y,0)}{|x_y-y|^{1/2}} \leq  \lambda.
$$
If otherwise $x_y\notin \partial K$, there exist $x_{y,1},\ldots,x_{y,d}\in K$ and $\mu_1,\ldots,\mu_d\in[0,1]$ such that
$$
\sum_{k=1}^d\mu_k=1 \quad\text{and}\quad \sum_{k=1}^{d}\mu_kx_{y,k}=x_y.
$$
Put $y_k=x_{y,k}+y-x_y$. Then also $\sum_{k=1}^{d}\mu_k y_k=y$ and thus
$$
v^*(y,0)\geq \min \{ v(y_1,0),\ldots,v(y_d,0)\}.
$$
Without loss, we may assume that the above minimum is attained at $v(y_1,0)$ and then we have
\begin{align*}
\frac{|v^*(y,0) - 1|}{\delta_{K^{\ast}}^{1/2}(y)}
=\frac{ 1 - v^*(y,0) }{|y-x_y|^{1/2}}
\leq  \frac{ 1- v(y_1,0) }{ |y_1 - x_{y, 1}|^{1/2}}
\leq\frac{|v(y_1,0) - 1|}{\delta_{K}^{1/2}(y_1)} 
\leq \lambda
\end{align*}
and thus $v^*$ satisfies \eqref{overdetermined}.

Now let us consider the solution $v_{K^*}$ of \eqref{problemv} associated to $K^*$: thanks to \eqref{subharmonic-convex} we have $v_{K^*}\geq v^*$ in $D$ and since $v_{K^*}=v^*=1$ on $\partial K^*$, we have that 
\eqref{overdetermined} for $v^*$ easily implies \eqref{overdetermined} for $v_{K^*}$, which concludes the proof.
\end{proof}

A straightforward corollary of the previous lemma is the following.
\begin{cor}
\label{lemma_uDlambda}
Assume that $\calF(D,\lambda)$ is not empty and define
$$
K_{D,\lambda}=\textrm{conv}\left(\bigcup_{K\in\calF(D,\lambda)}K\right)
$$
Then $K_{D,\lambda}\in\calF(D,\lambda)$. 
\end{cor}

For further convenience, let us denote as $u_{D,\lambda}$ the solution of \eqref{problemv} for $K_{D,\lambda}$ and notice that $u_{D,\lambda}$ satisfies \eqref{overdetermined} by the above corollary.

Following the proof of Lemma \ref{lemmasupv}, we also have the following.
\begin{lemma}\label{subharmonic}
Let $K\subset D$ be a nonempty compact set and assume there is $v\in C(\overline{D}\times \R)$ satisfying in the viscosity sense
\begin{equation}\label{problemvsub}
\left\{\begin{array}{ll}
\Delta v \geq 0\quad&\text{in }(D\times\R)\setminus(K\times\{0\}),\\
v\leq 0\quad&\text{on }\partial D\times\R,\\
v\leq 1\quad&\text{in }K\times\{0\},\\
\limsup_{|x|\to\infty} v(x)\leq  0.\quad&
\end{array}\right.
\end{equation}
If $v$ satisfies additionally  $\sup_{y\in D\setminus K} \frac{|v_{K}(y,0)-1|}{\delta_K^{1/2}(y)} \le \lambda$,
 then $K\in \cal{F}(D,\lambda)$.
\end{lemma}
\begin{proof}
Let $v$ satisfy the stated assumptions and let $v_K$ be the harmonic solution to \eqref{problemv} with this $K$ . Then it is enough to note by the Maximum Principle we have $v_K\geq v$ in $D\setminus K$ and thus 
$$
\frac{|v_K(y,0) - 1|}{\delta_{K}^{1/2}(y)}= \frac{ 1-v_K(y,0) }{\delta_{K}^{1/2}(y)}\leq \frac{ 1 - v(y,0) }{\delta_{K}^{1/2}(y)}\leq \lambda.
$$
Thus $K\in \cal{F}(D,\lambda)$ as claimed.
\end{proof}

\begin{lemma}
\label{C1boundary}
Assume that $\calF(D,\lambda)$ is not empty. Then $\partial K_{D,\lambda}$ is $C^1$.
\end{lemma}
\begin{proof}
The proof is similar to the proof of Lemma 2.6 in \cite{HS2000}. Let us abbreviate $K = K_{D,\lambda}$. 
 First, note that $K$ has nonempty interior. If not, since $K$ is convex, we get that $u_{D,\lambda} \equiv 0$ on $(\overline{D} \times \R) \setminus (K \times \{0\})$, which gives contradiction with (\ref{overdetermined}).
Next, we prove that there is a unique supporting plane at every point of $\partial K$. Put $v = 1 - u_{D,\lambda}$. Suppose that there is a point $x^0 \in \partial K$ such that we have 2 supporting planes $T_1$ and $T_2$. By rotation and translation we may assume that $x^0$ is the origin and 
$T_1 =\{x \in \R^d: \, x_1 + \eps x_2 = 0\}$, 
$T_2 =\{x \in \R^d: \, x_1 - \eps x_2 = 0\}$, 
$K \subset \{x \in \R^d: \, x_1 + \eps x_2 <  0\} \cap \{x \in \R^d: \, x_1 - \eps x_2 < 0\}$, for some $\eps > 0$. Now, for $i = 1, 2, 3$ define $\Lambda_i$ to be cones
$$
\Lambda_i =  \{x \in \R^d: \, x_1 + \eps x_2/i >  0\} \cup \{x \in \R^d: \, x_1 - \eps x_2/i > 0\}.
$$
Then $\Lambda_3 \subset \Lambda_2 \subset \Lambda_1 \subset K^c$ and $0 \in \partial K \cap \partial \Lambda_1$. By \cite[Theorem 3.2]{BB2004} there exists a function $w: \R^d \to [0,\infty)$, which is continuous on $\R^d$, positive and $1$-harmonic (i.e. harmonic with respect to $(-\Delta)^{1/2}$) on $\Lambda_3$, vanishes on $\Lambda_3^c$ and it is homogeneous of degree $\beta > 0$, that is $w(x) = |x|^{\beta} w(x/|x|)$ for any $x \in \R^d \setminus \{0\}$. It is the Martin kernel for $(-\Delta)^{1/2}$ with pole at infinity for $\Lambda_3$. By Lemma 3.3 and Example 3.2 in \cite{BB2004} $\beta < 1/2$. Note that this implies
\begin{equation}
\label{wlimit}
\limsup_{y \to 0} \frac{w(y)}{|y|^{1/2}} \ge \lim_{y \to 0} \frac{|y|^{\beta} w(e_1^d)}{|y|^{1/2}} = \infty.
\end{equation}
Now let $W$ be the harmonic extension in $\R^{d+1}$ of $w$, that is for $x \in \R^d$ we have $W(x,0) = w(x)$ and for $x \in \R^d$, $y \ne 0$ we have $W(x,y) = \int_{\R^d} P(x-z,|y|) w(z) \, dz$, where $P$ is given by (\ref{Poisson}). Choose $\delta > 0$ such that $B_{2 \delta}^d(0)\subset D$. Let
$$
S = \{x \in \R^{d+1}: \, x_1^2 + x_2^2 < 2 \delta^2/d, x_3^2, \ldots, x_d^2 \in [0,\delta^2/d), x_{d+1} \in (-1,1)\},
$$ 
$S^* = \Lambda_3^c \times \{0\}$. By properties of $w$  the function $W$ is harmonic in $S \setminus S^*$. Clearly, $v \ge W = 0$ on $S^*$. Obviously there exists a constant $c > 0$ such that $c v \ge W$ on $\partial S \cap (\Lambda_2 \times \R)$. On the other hand, by estimates of Poisson kernels (see e.g. \cite{K2005}) there exist constants $c_1 > 0$, $c_2 > 0$ such that $W(x) \le c_1 |x_{d+1}|$ for $x \in \partial S \times (\Lambda_2^c \times \R)$ and $v(x) \ge c_2 |x_{d+1}|$ for $x \in \partial S \times (\Lambda_2^c \times \R)$. Hence there exists $c_3 > 0$ such that $c_3 v \ge W$ on $\partial(S \setminus S^*) = \partial S \cup S^*$. By the comparison principle $c_3 v \ge W$ on $S \setminus S^*$. Using this and (\ref{wlimit}) we get
$$
\limsup_{y \to 0} \frac{|u_{D,\lambda}(y,0) - 1|}{|y|^{1/2}} \ge 
\limsup_{y \to 0} \frac{v(y)}{|y|^{1/2}} \ge
\limsup_{y \to 0} \frac{w(y)}{c_3 |y|^{1/2}} = \infty,
$$
which gives a contradiction. So, there is a unique supporting plane at every point $x \in \partial K$. The justification that these planes change continuously is the same as in the proof of Lemma 2.6 in \cite{HS2000}.
\end{proof}

For any nonzero vector $a \in \R^d$ and any $\varphi: \overline{D} \to \R$ denote
$$
\partial_{a}^{1/2} \varphi(\theta) = \lim\limits_{t\to 0^+} \frac{\varphi(\theta+t a) - \varphi(\theta)}{\sqrt{t}}  
$$
Let $\nu(\theta)$ denote the unit exterior normal vector of $\partial K_{D,\lambda}$ at $\theta \in \partial K_{D,\lambda}$. For $x \in \R^d$ put $\varphi_{D,\lambda}(x) = u_{D,\lambda}(x,0)$. 

For any $r > 0$ define  
\begin{equation}
\label{vr_def}
v_{r}(x) = \frac{1 - u_{D,\lambda}(r x)}{r^{1/2}}.
\end{equation}

\begin{lemma}
\label{rough_estimate}
Let $U \subset \R^{d+1}$ be a compact set. There exists $c =c(U, D, \lambda) > 0$ such that for sufficiently small $r \in (0,1]$ and all $x = (\tilde{x},x_{d+1}) \in U$ with $\tilde{x} \in \R^{d}$, $x_{d+1} \in \R \setminus \{0\}$ we have
$$
\left|v_{r}(x) - \int_{\R^d} P(\tilde{x} - y,|x_{d+1}|) v_{r}(y,0) \, dy \right|  \le
c r^{1/2} |x_{d+1}|.
$$
\end{lemma}
\begin{proof}
Assume that $r \in (0,1]$. Let $x^0$ be a point on $\partial K_{D,\lambda}$. We may assume that the origin is at $x^0$. Let $a = \dist(0,\partial D)$. By Lemma \ref{lemmadist} $a > 0$. Let $S \subset B_1^d(0) \times (0,1)$ be an open, convex set with $C^2$ boundary such that
$$
(\overline{B_{1/2}^d(0)} \times \{0\}) \cup (\overline{B_{1/2}^d(0))} \times \{1\}) \cup (\partial B_1^d(0) \times [1/4,3/4]) 
\subset \partial S
$$
and $\partial S \cap (\R^d \times \{0\}) = \overline{B_{1/2}^d(0)} \times \{0\}$.
For $t > 0$ put $S_t = \{tx: \, x \in S\}$. Note that $v_{r}$ is harmonic on $B_{a/r}^d(0) \times (0,a/r)$. Hence it is harmonic on $S_{a/r}$. 

For any open set $W \subset \R^{d+1}$ with $C^2$ boundary, $x \in W$, $y \in \partial W$ let $P_W(x,y)$ be the Poisson kernel of $W$ at $x$. By well known estimates of Poisson kernels (see e.g. \cite{K2005}) $P_S(x,y) \le c \dist(x,\partial S) |x - y|^{-d-1}$ for any $x \in S$, $y \in \partial S$. Hence, for any $x \in S_{a/r}$, $y \in \partial S_{a/r}$ we have
$$
P_{S_{a/r}}(x,y) = \left(\frac{r}{a}\right)^{d} P_S\left(\frac{x r}{a}, \frac{y r}{a}\right) \le \frac{c x_{d+1}}{|x - y|^{d+1}}.
$$
Let $S_{a/r}^* = \overline{B_{a/(2r)}^d(0)} \times \{0\}$. By the definition of $S$ we know that $S_{a/r}^* \subset \partial S_{a/r}$. For sufficiently small $r$ and any $x \in U$, $y \in \partial S_{a/r} \setminus S_{a/r}^*$ we have $|x - y| \ge a/(4 r)$. 
Let $\sigma_{a/r}$ be the surface measure on $\partial S_{a/r}$. It follows that for sufficiently small $r$, $x = (\tilde{x},x_{d+1}) \in U$  with $x_{d+1} > 0$ we have 
$$
v_{r}(x) =
\int_{\partial S_{a/r}} P_{S_{a/r}}(x,y) v_{r}(y) \, d\sigma_{a/r}(y)
= \int_{S_{a/r}^*} + \int_{\partial S_{a/r} \setminus S_{a/r}^*}
$$
The last integral is bounded from above by
\begin{equation}
\label{boundary_integral}
\int_{\partial S_{a/r} \setminus S_{a/r}^*} \frac{c x_{d+1} r^{-1/2}}{|x-y|^{d+1}} \, d\sigma_{a/r}(y)
\le c x_{d+1} r^{-1/2} (4 r/a)^{d+1} \sigma_{a/r}(\partial S_{a/r} \setminus S_{a/r}^*) \le
c a^{-1} x_{d+1} r^{1/2}.
\end{equation}
Hence,
\begin{equation}
\label{rough1}
\left|v_{r}(x) - \int_{S_{a/r}^*} P_{S_{a/r}}(x,y) v_{r}(y) \, d\sigma_{a/r}(y) \right|  \le
c a^{-1} x_{d+1} r^{1/2}.
\end{equation}
On the other hand we have
\begin{eqnarray}
\nonumber
&& \left|\int_{B_{a/(2r)}^d(0)} P(\tilde{x} - y, x_{d+1}) v_{r}(y,0) \, dy - \int_{S_{a/r}^*} P_{S_{a/r}}(x,y) v_{r}(y) \, d\sigma_{a/r}(y) \right|\\  
\nonumber
&& = \int_{\partial S_{a/r} \setminus S_{a/r}^*} P_{S_{a/r}}(x,y) \int_{B_{a/(2r)}^d(0)} P(\tilde{y} - z, y_{d+1}) v_{r}(z,0) 
\, dz \, d\sigma_{a/r}(y)\\
\nonumber
&& \le \int_{\partial S_{a/r} \setminus S_{a/r}^*} P_{S_{a/r}}(x,y) r^{-1/2} \, d\sigma_{a/r}(y)\\
\label{rough2}
&& \le
c a^{-1} x_{d+1} r^{1/2},
\end{eqnarray}
where in the last inequality we used (\ref{boundary_integral}). We used also the notation $y = (\tilde{y},y_{d+1})$.

We also have
$$
\int_{\R^d} P(\tilde{x} - y,x_{d+1}) v_{r}(y,0) \, dy - \int_{B_{a/(2r)}^d(0)} P(\tilde{x} - y, x_{d+1}) v_{r}(y,0) \, dy
\le c a^{-1} r^{1/2} x_{d+1}.
$$
Using this, (\ref{rough1}) and (\ref{rough2}) we obtain the assertion of the lemma.
\end{proof}

For any $x = (x_1,\ldots,x_d) \in \R^d$ put
\begin{equation*}
\psi(x) = \left\{
\begin{aligned}
x_1^{1/2},& \quad \text{for $x_1 > 0$,}\\
0,& \quad \text{for $x_1 \le 0$.}\\
\end{aligned}
\right.
\end{equation*}
Let $\Psi$ be the harmonic extension in $\R^{d+1}$ of $\psi$, that is, for $x \in \R^d$ we have $\Psi(x,0) = \psi(x)$ and for $x \in \R^d$, $y \ne 0$ we have
$\Psi(x,y)=\int_{\R^d}P(x-z,|y|)\psi(z)\ dz$. The explicit formula of $\Psi$ is well known see e.g. \cite[(1.3)]{SR2012}, however we will not use it in our paper.

\begin{lemma}
\label{subsequence}
Assume that $\calF(D,\lambda)$ is not empty and let $x^0$ be a point on $\partial K_{D,\lambda}$. Assume that the origin is at $x^0$ and that the exterior normal to $\partial K_{D,\lambda}$ at $x^0$ is directed by the first coordinate vector.  
Let $r_j$ be any decreasing sequence converging to $0$. Then there exists $\beta \ge 0$ and a subsequence also denoted by $r_{j}$ such that $v_{r_{j}}$ converges uniformly on any compact subset of $\R^{d+1}$ and pointwise on all of $\R^{d+1}$ to $\beta \Psi$.
\end{lemma}
\begin{proof} 
 For any $j$ and $x \in \R^d$ we have
\begin{equation}
\label{general_estimate}
v_{r_j}(x,0) 
= \frac{1 - u_{D,\lambda}(r_j x,0)}{r_j^{1/2}} 
\le \frac{ \lambda r_j^{1/2}|x|^{1/2}}{r_j^{1/2}} 
=  \lambda |x|^{1/2}.
\end{equation}

Denote $\R_-^d = \{x \in \R^d: \, x_1 < 0\}$. Let $U \subset \R^d$ be a convex, compact set such that $0 \in  U$. Let $J_0$ be such that for any $j \ge J_0$ we have $4 \overline{U} \subset r_j^{-1} D$ and for any $x \in U \cap \R_+^d$ we have $\dist(x,\partial(r_j^{-1} D \cap \R_+^d)) = \dist(x,\partial \R_+^d) = x_1$. Now, we will show that the family $\{v_{r_j}(x,0)\}_{j \ge J_0}$ is uniformly equicontinuous on $U$. Since $U$ is compact it is enough to show that this family is equicontinuous at each $x \in U$.

First we show it for $x \in U \cap \R_+^d$. Note that for $j \ge J_0$ functions $v_{r_j}$ are harmonic in $(r_j^{-1} D \cap \R_+^d) \times (0,\infty)$. Hence for $j \ge J_0$ and $x \in U \cap \R_+^d$ we have
$$
|\nabla v_{r_j}(x,0)| \le (d+1) \frac{v_{r_j}(x,0)}{\dist(x,\partial(r_j^{-1} D \cap \R_+^d))} \le \frac{ \lambda |x|^{1/2}}{x_1}.
$$
Therefore, $\{v_{r_j}(x,0)\}_{j \ge J_0}$ is equicontinuous at each $x \in U \cap \R_+^d$.

Let $x \in \R_-^d$. Then there exists $J_1 \ge J_0$ ($J_1$ depends on $x$) such that for any $j \ge J_1$ we have $x \in \text{int}(r_j^{-1} K_{D,\lambda})$ so $v_{r_j}(y,0) = 1$ for $y$ in some ball in $\R^d$ with centre at $x$. Hence $\{v_{r_j}(x,0)\}_{j \ge J_0}$ is equicontinuous at each $x \in U \cap \R_-^d$.

For $x = 0$ and $y \in U$ we have $|v_{r_j}(x,0) - v_{r_j}(y,0)|/|x - y|^{1/2} = v_{r_j}(y,0)/|y|^{1/2} \le  \lambda$.

Now, let us consider the case when $x \in U \cap \partial \R_+^d$, $x \ne 0$. For $y \in \R^d$ we put $y = (y_1,y_*)$, where $y_1 \in \R$, $y_* \in \R^{d-1}$. For any $r > 0$ let $W_r = \{y = (y_1,y_*): \, |y_1| \le r, |y_*| \le r\}$. By Lemma \ref{C1boundary} $\partial K_{D,\lambda}$ is locally a graph of a $C^1$ function. More precisely, there exists $R > 0$ and $f: B_R^{d-1}(0) \to (-\infty,0]$ such that 
$$
\partial K_{D,\lambda} \cap W_R = \{y = (y_1,y_*) \in W_r: y_1 = f(y_*)\,\}
$$
and $\lim_{y_* \to 0} f(y_*)/|y_*| = 0$. Note that there exists $J_2 \ge J_0$ such that if $j \ge J_2$ and $x \in U$ then $|r_j x_*| \le R$. For $x \in U \cap \partial \R_+^d$, $x \ne 0$ and $j \ge J_2$ we have
$$
v_{r_j}(x,0) = \frac{1 - u_{D,\lambda}(r_j x,0)}{r_j^{1/2}} \le
\frac{ \lambda \dist^{1/2}(r_j x, K_{D,\lambda})}{r_j^{1/2}} \le
\frac{ \lambda |f(r_j x_*)|^{1/2}}{r_j^{1/2}}.
$$
Note that $\lim_{j \to \infty} |f(r_j x_*)|^{1/2}/r_j^{1/2} = 0$. Similarly, for $y \in U$ and $j \ge J_2$ we have
$$
v_{r_j}(y,0) = \frac{1 - u_{D,\lambda}(r_j y,0)}{r_j^{1/2}} \le
\frac{ \lambda \dist^{1/2}(r_j y, K_{D,\lambda})}{r_j^{1/2}} \le
\frac{ \lambda (|f(r_j  y_*)| + r_j|y_1|)^{1/2}}{r_j^{1/2}} \le
\frac{ \lambda |f(r_j  y_*)|^{1/2}}{r_j^{1/2}} +  \lambda|y_1|^{1/2}.
$$
Therefore, $\{v_{r_j}(x,0)\}_{j \ge J_0}$ is equicontinuous at each $x \in U \cap \partial \R_+^d$. 

So finally, $\{v_{r_j}(x,0)\}_{j \ge J_0}$ is uniformly equicontinuous and uniformly bounded on $U$. Since $0 \in U \subset \R^d$ was an arbitrary compact, convex set there is a subsequence of $\{v_{r_j}(x,0)\}$ also denoted by $\{v_{r_j}(x,0)\}$ which converges uniformly on any bounded subset of $\R^{d}$ and pointwise on all of $\R^{d}$ to a continuous function $v_0(x,0)$. Clearly, $0 \le v_0(x,0) \le c |x|^{1/2}$. 

For $x \in \R^d$, $y \ne 0$ let us define 
\begin{equation}
\label{v0rep}
v_{0}(x,y) = \int_{\R^d} P(x - z,|y|) v_{0}(z,0) \, dz.
\end{equation}
Note that $v_0$ is continuous on $\R^{d+1}$. By  Lemma \ref{rough_estimate} $\{v_{r_j}\}$ converges uniformly on any bounded subset of $\R^{d+1}$ and pointwise on all of $\R^{d+1}$ to $v_0$.

Note that $v_0$ is harmonic on $\R_+^d \times \R$ so is $C^{\infty}$ on $\R_+^d \times \R$. We also have $v_0(x,-y) = v_0(x,y)$ for any $x \in \R^d$, $y > 0$ so 
\begin{equation}
\label{v0derivative}
\frac{\partial v_0}{\partial x_{d+1}}(x,0) = 0, \quad \text{for} \quad x \in \R_+^d.
\end{equation}
For any $x \in \R^d$ put $\varphi_0(x) = v_0(x,0)$. By (\ref{v0rep}) and (\ref{v0derivative}) $\varphi_0$ is $1/2$-harmonic on $\R_+^d$ (see e.g. \cite{CS2007}). We also have $\varphi_0 \equiv 0$ on $\overline{\R_-^d}$. By \cite[Theorem 4]{BKK2008} and \cite[Theorem 3.2 and Example 3.2]{BB2004} we get $\varphi_0 \equiv \beta \psi$ on $\R^d$ for some $\beta \ge 0$. 
\end{proof}

Recall that $e_1^d = (1,0,\ldots,0) \in \R^d$. Let $j:\R^d \to [0,1]$ be a continuous function on $\R^d$ such that $j \equiv 1$ on $B_1^d(-e_1^d)$, $(-\Delta)^{1/2}j \equiv 0$ on $(\overline{B_1^d(-e_1^d)})^c$ and $\lim_{|x| \to \infty} j(x) = 0$. For $x \in \R^d$ define $q_1(x) = 1 - j(x)$. Let $Q_1$ be a harmonic extension of $q_1$ on $\R^{d+1}$. 
For $y \in (B_1^d(0))^c$ define
$$
I(y) = \frac{\Gamma((d-1)/2)}{2 \pi^{1/2} \Gamma(d/2)} \frac{1}{|y|^{d-2}} \int_0^{1/(|y|^2 - 1)} \frac{(1 - (|y|^2 - 1)b)^{(d-2)/2}}{b^{1/2}(1+b)} \, db.
$$
By Appendix in \cite{L1972} (cf. also (3.3), (3.6) in \cite{K1997}) we have
$$
I(e_1^d) = \frac{\Gamma((d-1)/2)}{2 \pi^{1/2} \Gamma(d/2)} \int_0^{\infty} \frac{1}{b^{1/2}(1+b)} \, db = 
\frac{\pi^{1/2} \Gamma((d-1)/2)}{2 \Gamma(d/2)}
$$
and for $y \in (B_1^d(-e_1^d))^c$
\begin{equation}
\label{j_formula}
j(y) = \frac{I(y+e_1^d)}{I(e_1^d)}.
\end{equation}

\begin{lemma}
\label{q1_fractional_derivative}
We have
$$
\partial_{e_1^d}^{1/2}q_1(0)=\lim_{t\to0^+}\frac{q_1(te_1^d)}{\sqrt{t}}=C_0:=\frac{2\sqrt{2}\Gamma(\frac{d}{2})}{\sqrt{\pi}\Gamma(\frac{d-1}{2})} .
$$
\end{lemma}
\begin{proof}
First note that with $x=te_1^d$ for $t>0$ we have
\begin{align*}
q_1(te_1^d)&=1-\frac{2\Gamma(\frac{d}{2})}{\sqrt{\pi}\Gamma(\frac{d-1}{2})}I((1+t)e_1^d)=1-\frac{(1+t)^{2-d}}{\pi}\int_0^{\frac{1}{t(2+t)}}\frac{(1-(2+t)tb)^{\frac{d}{2}-1}}{b^{1/2}(1+b)} \ db.
\end{align*}
In the following, we write for functions $f,g:(0,1)\to (0,\infty)$
$$
g(t)\sim f(t)\quad \text{for $t\to 0^+$},\qquad\text{ if }\quad \lim_{t\to0} \frac{g(t)}{f(t)}=1.
$$
With this notation and a change of variable we find for $t\to 0^+$
\begin{align*}
\frac{q_1(te_1^d)}{\sqrt{t}}
&= \frac{1}{\sqrt{t}}\Bigg(1-\frac{(1+t)^{2-d}(2+t)^{1/2}}{\pi}\int_0^{\frac{1}{t}}\frac{(1- tb)^{\frac{d}{2}-1}}{b^{1/2}((2+t)+b)} \ db\Bigg)\\
&\sim \frac{1}{\sqrt{t}}\Bigg(1-\frac{ \sqrt{2}}{\pi}\int_0^{\frac{1}{t}}\frac{(1- tb)^{\frac{d}{2}-1}}{b^{1/2}(2+t+b)} \ db\Bigg),
\end{align*}
where we have used that
\begin{align*}
 \lim_{t\to0^+}&\Bigg(\frac{(1+t)^{2-d}(2+t)^{1/2}-\sqrt{2}}{\pi\sqrt{t}}\int_0^{\frac{1}{t}}\frac{(1- tb)^{\frac{d}{2}-1}}{b^{1/2}((2+t)+b)} \ db\Bigg)\\
&= \frac{1}{\pi}\int_0^{\infty}\frac{1}{b^{1/2}(2+b)} \ db\lim_{t\to0^+} \sqrt{t}\frac{(1+t)^{2-d}(2+t)^{1/2}-\sqrt{2}}{ t}=0.
\end{align*}
Using moreover that
$$
1=\frac{\sqrt{2}}{\pi}\int_0^{\infty}\frac{1}{b^{1/2}(2+b)}\ db=\lim_{t\to0^+} \frac{\sqrt{2}}{\pi}\int_0^{\infty}\frac{1}{b^{1/2}(2+t+b)}\ db
$$
and
\begin{align*}
\lim_{t\to0^+} \frac{1}{\sqrt{t}}\Bigg(1-\frac{\sqrt{2}}{\pi}\int_0^{\infty}\frac{1}{b^{1/2}(2+t+b)}\ db\Bigg)&=\frac{\sqrt{2}}{\pi}\lim_{t\to0^+}\int_0^{\infty}\frac{\sqrt{t}}{b^{1/2}} \frac{\frac{1}{2+b}-\frac{1}{2+t+b}}{t}\ db\\
&=\frac{\sqrt{2}}{\pi}\lim_{t\to0^+}\sqrt{t}\int_0^{\infty}\frac{1}{b^{1/2}(2+b)^2}\ db=0,
\end{align*}
it follows that we have
\begin{align*}
\frac{q_1(te_1^d)}{\sqrt{t}}&\sim \frac{\sqrt{2}}{\pi\sqrt{t}} \Bigg(\int_{1/t}^{\infty}\frac{1}{b^{1/2}(2+t+b)}\ db + \int_0^{\frac{1}{t}}\frac{1-(1- tb)^{\frac{d}{2}-1}}{b^{1/2}(2+t+b)} \ db\Bigg) \quad\text{for $t\to 0^+$.}
\end{align*}
For the first integral we have the limit
\begin{equation}\label{one limit controlled}
\lim_{t\to0^+}\frac{1}{\sqrt{t}}\int_{1/t}^{\infty}\frac{1}{b^{1/2}(2+t+b)}\ db=\lim_{t\to0^+} \int_{1}^{\infty}\frac{1}{s^{1/2}(2t+t^2+s)}\ ds=2.
\end{equation}
If $d=2$, then the second integrals vanishes
%$$
%\int_0^{\frac{1}{t}}\frac{1-(1- tb)^{\frac{d}{2}-1}}{b^{1/2}(2+t+b)} \ db=0,
%$$
and thus $\lim_{t\to0^+}q_1(te_1^2)/\sqrt{t}=2\sqrt{2}/\pi$.\\
If $d>2$, then
\begin{align*}
 \lim_{t\to0^+} \frac{1}{\sqrt{t}}\int_0^{\frac{1}{t}}\frac{1-(1- tb)^{\frac{d}{2}-1}}{b^{1/2}(2+t+b)} \ db\,
&= \lim_{t\to0^+} \int_0^{1}\frac{1-(1- s)^{\frac{d}{2}-1}}{s^{1/2}(2t+t^2+s)} \ ds\\
&=\int_0^{1}\frac{1-(1-s)^{\frac{d}{2}-1}}{s^{3/2}}\ ds\\
&=\frac{-2(1-(1-s)^{\frac{d}{2}-1})}{s^{1/2}}\Bigg|_0^1+(d-2)\int_0^{1} (1-s)^{\frac{d}{2}-2} s^{-\frac{1}{2}}\ ds\\
&=-2+\frac{(d-2)\Gamma(\frac{d}{2}-1)\sqrt{\pi}}{\Gamma(\frac{d-1}{2})}=-2+ \frac{2\sqrt{\pi}\Gamma(\frac{d}{2})}{\Gamma(\frac{d-1}{2})}.
\end{align*}
And thus, with \eqref{one limit controlled} the claim follows also for $d>2$.
\end{proof}

\begin{lemma}\label{Q1 satisfies boundary condition}
It holds
\begin{equation}\label{Q1 to show}
Q_1(y,0)\leq C_0 \dist(y,B_1^d(-e_1^d))^{1/2} =\frac{2\sqrt{2}\Gamma(\frac{d}{2})}{\sqrt{\pi}\Gamma(\frac{d-1}{2})} \dist(y,B_1^d(-e_1^d))^{1/2} \quad\text{for $y\in \R^d$.}
\end{equation}
\end{lemma}
\begin{proof}
Recall that $Q_1$ is the harmonic extension of $q_1$ in $\R^{d+1}$, where $q_1=0$ on $\overline{B_1^d(-e_1^d)}$ and, for $y\in (\overline{B_1^d(-e_1^d)})^c$, it is given by
\begin{equation}
q_1(y)=1-j(y)=1-\frac{|y+e_1^d|^{2-d}}{\pi}\int_0^{\frac{1}{|y+e_1^d|^2-1}}\frac{(1-(|y+e_1^d|^2-1)b)^{\frac{d-2}{2}}}{b^{1/2}(1+b)}\ db.
\end{equation}
With $x=y+e_1^d$ it is thus enough to show that for $|x|>1$ we have
\begin{equation}\label{to show2}
\begin{split}
 \frac{2\sqrt{2}\Gamma(\frac{d}{2})}{\sqrt{\pi}\Gamma(\frac{d-1}{2})}(|x|-1)^{1/2}&\geq 1-\frac{|x|^{2-d}}{\pi}\int_0^{\frac{1}{|x|^2-1}}\frac{(1-(|x|^2-1)b)^{\frac{d-2}{2}}}{b^{1/2}(1+b)}\ db\\
&=1-\frac{|x|^{2-d}(|x|^2-1)^{1/2}}{\pi}\int_0^{1}\frac{(1-\tau)^{\frac{d-2}{2}}}{\tau^{1/2}(|x|^2-1+\tau)}\ d\tau.
\end{split}
\end{equation}
\textit{The case $d=2$:} In this case \eqref{to show2} reads
\begin{align*}
 \frac{2\sqrt{2}}{\pi}(|x|-1)^{1/2}&\geq 1-\frac{\sqrt{|x|^2-1 }}{\pi}\int_0^{1}\frac{1}{\tau^{1/2}(|x|^2-1+\tau)}\ d\tau=1-\frac{2 }{\pi}\arctan(\frac{1}{\sqrt{|x|^2-1}}),
\end{align*}
or, equivalently, with $a=|x|$, 
$$
f(a):=\sqrt{2(a-1)}+\arctan(\frac{1}{\sqrt{a^2-1}})-\frac{\pi}{2}\geq 0 \quad\text{for $a>1$.}
$$ 
Note that the function $f$ can be extended continuously at $1$ with $f(1)=0$ and 
\begin{align*}
f'(a)&=\frac{1}{\sqrt{2(a-1)}}-\frac{1}{a\sqrt{a^2-1}}=\frac{1}{\sqrt{a-1}}\Big(\frac{1}{\sqrt{2}}-\frac{1}{a\sqrt{a+1}}\Big) >0 \quad\text{for $a>1$.}
\end{align*}
This shows the claim for $d=2$.

\textit{The case $d\geq 3$:} Denote
$$
f: (1,\infty) \to \R, \quad f(a)=\frac{2\sqrt{2}\Gamma(\frac{d}{2})}{\sqrt{\pi}\Gamma(\frac{d-1}{2})}(a-1)^{1/2}-1+\frac{a^{2-d}}{\pi}\int_0^{\frac{1}{a^2-1}}\frac{(1-(a^2-1)b)^{\frac{d-2}{2}}}{b^{1/2}(1+b)}\ db.
$$
Note that $\lim_{a \to 1^+}f(a)=-1 + \frac{1}{\pi} \int_0^{\infty} \frac{db}{b^{1/2}(1+b)} = 0$, so that it is enough to show that $f'>0$. For the following calculation recall
$$
\int_0^1t^{x-1}(1-t)^{y-1}\ dt=\frac{\Gamma(x)\Gamma(y)}{\Gamma(x+y)}.
$$
Then for $a>1$ we have
\begin{align*}
f'(a)&=\frac{ \sqrt{2}\Gamma(\frac{d}{2})}{\sqrt{\pi}\Gamma(\frac{d-1}{2})\sqrt{a-1}}-(d-2)\frac{a^{1-d}}{\pi}\int_0^{\frac{1}{a^2-1}}\frac{(1-(a^2-1)b)^{\frac{d-2}{2}}}{b^{1/2}(1+b)}\ db\notag\\
&\qquad -(d-2)\frac{a^{3-d}}{\pi}\int_0^{\frac{1}{a^2-1}}\frac{b^{1/2}(1-(a^2-1)b)^{\frac{d-4}{2}}}{(1+b)}\ db\\
&\geq 
%\frac{ \sqrt{2}\Gamma(\frac{d}{2})}{\sqrt{\pi}\Gamma(\frac{d-1}{2})\sqrt{a-1}}- \frac{d-2}{\pi}\int_0^{\frac{1}{a^2-1}}\frac{(1-(a^2-1)b)^{\frac{d-2}{2}}}{b^{1/2}(1+b)}\ db\notag\\
%&\qquad -\frac{d-2}{\pi}\int_0^{\frac{1}{a^2-1}}\frac{b^{1/2}(1-(a^2-1)b)^{\frac{d-4}{2}}}{(1+b)}\ db\\
%&=
\frac{ \sqrt{2}\Gamma(\frac{d}{2})}{\sqrt{\pi}\Gamma(\frac{d-1}{2})\sqrt{a-1}}- \frac{(d-2)\sqrt{a^2-1}}{\pi}\int_0^{1}\frac{(1- t)^{\frac{d-2}{2}}}{t^{1/2}(a^2-1+t)}\ dt\notag\\
&\qquad -\frac{d-2}{\pi\sqrt{a^2-1}}\int_0^{1}\frac{t^{1/2}(1-t)^{\frac{d-4}{2}}}{(a^2-1+t)}\ dt\\
%&=\frac{ \sqrt{2}\Gamma(\frac{d}{2})}{\sqrt{\pi}\Gamma(\frac{d-1}{2})\sqrt{a-1}}- \frac{(d-2)}{\pi\sqrt{a^2-1}}\int_0^{1}\frac{(1- t)^{\frac{d-4}{2}}}{t^{1/2}(a^2-1+t)}\Big((a^2-1)(1-t)+t)\ dt\notag\\
%&=\frac{ \sqrt{2}\Gamma(\frac{d}{2})}{\sqrt{\pi}\Gamma(\frac{d-1}{2})\sqrt{a-1}}- \frac{(d-2)}{\pi\sqrt{a^2-1}}\int_0^{1}\frac{(1- t)^{\frac{d-4}{2}}}{t^{1/2}(a^2-1+t)}\Big((a^2-1+t)(1-t)-t(1-t)+t)\ dt\notag\\
&=\frac{ \sqrt{2}\Gamma(\frac{d}{2})}{\sqrt{\pi}\Gamma(\frac{d-1}{2})\sqrt{a-1}}- \frac{(d-2)}{\pi\sqrt{a^2-1}}\int_0^{1}\frac{(1- t)^{\frac{d-4}{2}}}{t^{1/2}(a^2-1+t)}\Big((a^2-1+t)(1-t) +t^2 \Big)\ dt\notag\\
%&=\frac{ \sqrt{2}\Gamma(\frac{d}{2})}{\sqrt{\pi}\Gamma(\frac{d-1}{2})\sqrt{a-1}}- \frac{(d-2)}{\pi\sqrt{a^2-1}}\int_0^{1}\frac{(1- t)^{\frac{d-2}{2}}}{t^{1/2}} \ dt- \frac{(d-2)}{\pi\sqrt{a^2-1}}\int_0^{1}\frac{t^{3/2}(1- t)^{\frac{d-4}{2}}}{(a^2-1+t)}\ dt\notag\\
%&=\frac{ \sqrt{2}\Gamma(\frac{d}{2})}{\sqrt{\pi}\Gamma(\frac{d-1}{2})\sqrt{a-1}}- \frac{(d-2)}{\pi\sqrt{a^2-1}}\frac{\sqrt{\pi}\Gamma(\frac{d}{2})}{\Gamma(\frac{d+1}{2})}- \frac{(d-2)}{\pi\sqrt{a^2-1}}\int_0^{1}\frac{t^{3/2}(1- t)^{\frac{d-4}{2}}}{(a^2-1+t)}\ dt\notag\\
&\geq\frac{ \sqrt{2}\Gamma(\frac{d}{2})}{\sqrt{\pi}\Gamma(\frac{d-1}{2})\sqrt{a-1}}- \frac{(d-2)}{\pi\sqrt{a^2-1}}\frac{\sqrt{\pi}\Gamma(\frac{d}{2})}{\Gamma(\frac{d+1}{2})}- \frac{(d-2)}{\pi\sqrt{a^2-1}}\int_0^{1} t^{1/2}(1- t)^{\frac{d-4}{2}} \ dt\notag\\
%&=\frac{ \sqrt{2}\Gamma(\frac{d}{2})}{\sqrt{\pi}\Gamma(\frac{d-1}{2})\sqrt{a-1}}- \frac{2(d-2)\Gamma(\frac{d}{2})}{\sqrt{\pi}(d-1)\Gamma(\frac{d-1}{2})\sqrt{a^2-1}}- \frac{(d-2)}{\pi\sqrt{a^2-1}}\int_0^{1}\frac{t^{3/2}(1- t)^{\frac{d-4}{2}}}{(a^2-1+t)}\ dt\notag\\
%&\geq\frac{ \sqrt{2}\Gamma(\frac{d}{2})}{\sqrt{\pi}\Gamma(\frac{d-1}{2})\sqrt{a-1}}- \frac{2(d-2)\Gamma(\frac{d}{2})}{\sqrt{\pi}(d-1)\Gamma(\frac{d-1}{2})\sqrt{a^2-1}}- \frac{(d-2)}{\pi\sqrt{a^2-1}}\int_0^{1} t^{1/2}(1- t)^{\frac{d-4}{2}} \ dt\notag\\
&=\frac{ \sqrt{2}\Gamma(\frac{d}{2})}{\sqrt{\pi}\Gamma(\frac{d-1}{2})\sqrt{a-1}}- \frac{2(d-2)\Gamma(\frac{d}{2})}{\sqrt{\pi}(d-1)\Gamma(\frac{d-1}{2})\sqrt{a^2-1}}- \frac{(d-2)}{\pi\sqrt{a^2-1}}\frac{\sqrt{\pi}\Gamma(\frac{d}{2}-1)}{2\Gamma(\frac{d+1}{2})}\notag\\
%&=\frac{\Gamma(\frac{d}{2})}{\sqrt{\pi}\Gamma(\frac{d-1}{2})\sqrt{a-1}}\Big(\sqrt{2}-\frac{2(d-2)}{(d-1)\sqrt{a+1}}- \frac{2}{(d-1)\sqrt{a+1}}\Big)\\
&=\frac{\Gamma(\frac{d}{2})}{\sqrt{\pi}\Gamma(\frac{d-1}{2})\sqrt{a-1}}\Big(\sqrt{2}-\frac{2}{\sqrt{a+1}}\Big)> 0,
\end{align*}
This shows the claim. 
\end{proof}

For $x = (x_1,x_2,\ldots,x_d,x_{d+1}) \in \R^{d+1}$ or $x = (x_1,x_2,\ldots,x_d) \in \R^{d}$ we put $x_* = (x_2,\ldots,x_{d})$ and $|x_*| = (x_2^2 + \ldots + x_{d}^2)^{1/2}$.
\begin{equation}
\label{def:kappa-ast}
\text{For  $x_* \in \R^{d-1}$ with $|x_*| < 1$ let $\kappa(x_*) \in (-1,0]$ be such that $(1+\kappa(x_*))^2 + |x_*|^2 = 1$.}
\end{equation}

By Lemma \ref{q1_fractional_derivative} and the fact that $\R^d \ni x \to q_1(x - e_1^d)$ is radial we obtain
\begin{lemma}
\label{Qeps}
For any $\eps > 0$ there exists $r \in (0,1)$ such that for any  $(x_1,x_*,x_{d+1}) \in [-r,r]^{d+1}$ we have
$$
\frac{Q_1(x_1,x_*,x_{d+1})}{C_0} \ge (1-\eps) \Psi(x_1-\kappa(x_*),0,x_{d+1}) \ge (1-\eps) \Psi(x_1,0,x_{d+1}).
$$ 
\end{lemma}

In the sequel we use the following notation $a \vee b = \max(a,b)$, $a \wedge b = \min(a,b)$.
\begin{lemma}
\label{Poisson1}
There exists $C_1 = C_1(d)$ such that for any $x \in \R^{d+1}$, with $x_{d+1} > 0$ we have
\begin{equation}
\label{Poisson1_ineq}
\int_{\R_+^d} P((x_1-y_1,x_*-y_*),x_{d+1}) y_1^{1/2} \, dy_1 \, dy_* \ge
\frac{ C_1 x_{d+1}}{((-x_1) \vee x_{d+1})^{1/2}}.
\end{equation}
\end{lemma}
\begin{proof}
Put $r = (-x_1) \vee x_{d+1}$. For $x_1 < 0$ the left-hand side of (\ref{Poisson1_ineq}) is bounded from below by
$$
c \int_{B_r^d((0,x_*)) \cap \R_+^d} \frac{x_{d+1} y_1^{1/2}}{r^{d+1}} \, dy_1 \, dy_* \ge
\frac{c x_{d+1}}{r^{1/2}}.
$$
For $x_1 \ge 0$ the left-hand side of (\ref{Poisson1_ineq}) is bounded from below by
$$
c \int_{B_{r}^d((x_1,x_*)) \cap \R_+^d} \frac{x_{d+1} y_1^{1/2}}{r^{d+1}} \, dy_1 \, dy_* \ge
\frac{c x_{d+1}}{r^{1/2}}.
$$
\end{proof}

\begin{lemma}
\label{Poisson2}
There exists $C_2 = C_2(d)$ such that for any $R > 0$, $|x_1| \le R/4$, $|x_*| \le R/4$, $x_{d+1} \in (0,R/4]$  we have
\begin{equation}
\label{Poisson2_ineq}
\int_{(B_R^d(0))^c} P((x_1-y_1,x_*-y_*),x_{d+1}) |y|^{1/2} \, dy_1 \, dy_* \le
\frac{ C_2 x_{d+1}}{R^{1/2}}.
\end{equation}
\end{lemma}
\begin{proof}
The left-hand side of (\ref{Poisson2_ineq}) is bounded from above by
$$
c \int_{(B_R^d(0))^c} \frac{x_{d+1} |y|^{1/2}}{|y|^{d+1}} \, dy \le \frac{c x_{d+1}}{R^{1/2}}.
$$
\end{proof}

\begin{lemma}
\label{Poisson3}
There exists $C_3 = C_3(d)$ such that for any $x \in \R^{d+1}$ with $x_1 < 0$ and $x_{d+1} \in (0,|x_1|]$ we have
\begin{equation}
\label{Poisson3_ineq}
\int_{(B_{|x_1|/2}^d(x))^c} P((x_1-y_1,x_*-y_*),x_{d+1}) \, dy_1 \, dy_* \le
\frac{ C_3 x_{d+1}}{|x_1|}.
\end{equation}
\end{lemma}
\begin{proof}
The left-hand side of (\ref{Poisson3_ineq}) is bounded from above by
$$
c \int_{(B_{|x_1|/2}^d(x))^c} \frac{x_{d+1} }{|y - x|^{d+1}} \, dy \le \frac{c x_{d+1}}{|x_1|}.
$$
\end{proof}

The proof of next lemma is using ideas from the proof of Theorem 2.1 in \cite{HS2000}. 
\begin{lemma}
\label{normal_derivative}
Assume that $\calF(D,\lambda)$ is not empty. Then for any $\theta \in \partial K_{D,\lambda}$ we have 
$\partial_{\nu(\theta)}^{1/2} \varphi_{D,\lambda}(\theta) = -\lambda$. 
\end{lemma}
\begin{proof}
Assume, for a contradiction, that there exists a point $\theta \in \partial K_{D,\lambda}$ such that $\partial_{\nu(\theta)}^{1/2} \varphi_{D,\lambda}(\theta)$ does not exist or it exists but $\partial_{\nu(\theta)}^{1/2} \varphi_{D,\lambda}(\theta) \ne -\lambda$.  Then $\limsup_{t \to 0^+} (\varphi_{D,\lambda}(\theta + t \nu(\theta)) - 1)/t^{1/2} = -\beta$ for some $0 \le \beta < \lambda$. Assume that the origin is at $\theta$ and that the exterior normal to $\partial K_{D,\lambda}$ at $\theta$ is directed by the first coordinate vector $e_1^d$. Then there exist a decreasing sequence $r_j$ tending to $0$ such that $r_j \in (0,1]$ for any $j$ and $\lim_{j \to \infty} (1 - \varphi_{D,\lambda}(r_j e_1^d))/r_j^{1/2} = \beta$ For any $r > 0$ let $v_r$ be defined by (\ref{vr_def}). By Lemma \ref{subsequence} there exists a decreasing subsequence of $r_j$ also denoted by $r_j$ such that the sequence $v_{r_j}$ converges uniformly on any compact subset of $\R^{d+1}$ and pointwise on all of $\R^{d+1}$ to $\beta \Psi$.

\begin{figure}
\centering
\includegraphics[scale=0.8]{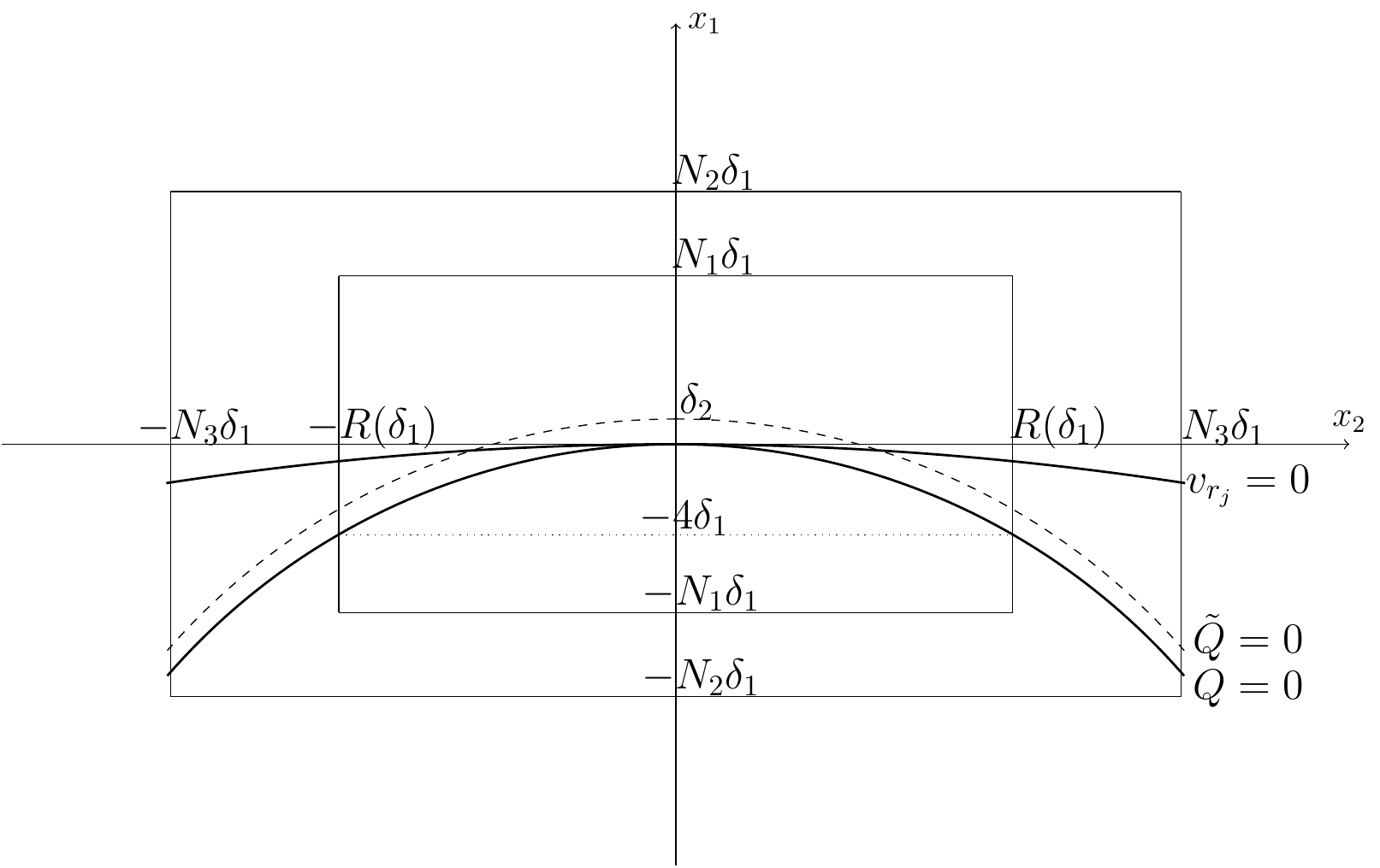}
\caption{Boxes $W_1$, $W_2$.}
\label{fig:2}
\end{figure}

For any $\delta_1 \in (0,1/16)$ put $R(\delta_1) = (1-(1-4\delta_1)^2)^{1/2} = (8\delta_1 - 16 \delta_1^2)^{1/2}$ and note that $R(\delta_1)\in(\sqrt{7\delta_1}, \frac{\sqrt{7}}{4})$. Let $\delta_1 \in (0,1/16)$ and $N_1 \ge 4$ be such that $R(\delta_1) \ge 2 N_1 \delta_1$. Let $N_2, N_3 \ge 4 R(\delta_1)/\delta_1$ be such that $N_2 \delta_1, N_3\delta_1 \le 3/4$. Note that thus $N_2,N_3>N_1$.
We define a small box
$$
W_1 = \{(x_1,x_*,x_{d+1}) \in \R^{d+1}: \, x_1 \in [-N_1 \delta_1, N_1 \delta_1], \, |x_*| \le R(\delta_1), \, x_{d+1} \in [-N_1 \delta_1, N_1 \delta_1]\}
$$
and a large box
$$
W_2 = \{(x_1,x_*,x_{d+1}) \in \R^{d+1}: \, x_1 \in [-N_2 \delta_1, N_2 \delta_1], \, |x_*| \le N_3 \delta_1, \, x_{d+1} \in [-N_2 \delta_1, N_2 \delta_1]\}
$$
(see Figure \ref{fig:2}).

Choose $\eps_1 > 0$ such that $\beta +\eps_1 < \lambda^{1/2}$. Put $Q = (\beta + \eps_1) Q_1/C_0$. Let $\eps_2 \in (0,\eps_1]$. 
We choose $\delta_1$ small enough such that we may choose $N_1, N_2, N_3$ so that 
\begin{equation}
\label{N2N3}
\frac{N_2 \wedge N_3}{N_1} \ge \frac{36 \lambda^2 C_2^2}{C_1^2 \eps_2^2}.
\end{equation}
By Lemma \ref{Qeps} one can choose $\delta_1, N_1, N_2, N_3$ (satisfying the conditions mentioned above) such that 
$$
Q(x_1,x_*,x_{d+1}) \ge (\beta  + \eps_2) \Psi(x_1-\kappa(x_*),0,x_{d+1}) \ge (\beta  + \eps_2) \Psi(x_1,0,x_{d+1})
$$ 
for any $(x_1,x_*,x_{d+1}) \in W_2$. From here on, $N_1,N_2,N_3,\delta_1,\epsilon_1,\epsilon_2$ are fixed such that the above inequalities holds.

Let $\delta_2 \in (0,\delta_1)$ and $\tilde{Q}(x_1,x_*,x_{d+1}) = Q(x_1 - \delta_2,x_*,x_{d+1})$. Note that $(1-4\delta_1)^2 + (R(\delta_1))^2 = 1$ so $Q( -4\delta_1, x_*,0) = 0$ for $x_* \in \R^{d-1}$ such that $|x_*| = R(\delta_1)$. Hence for such $x_*$ we have 
$\tilde{Q}( -4\delta_1+\delta_2, x_*,0) = 0$.

Let $\eps_3 > 0$. We choose $j$ large enough so that
\begin{equation}
\label{beta_eps3} 
v_{r_j}(y) \le \beta \Psi(y) + \eps_3, \quad y \in W_2, 
\end{equation}
\begin{equation}
\label{delta1_belt}
\{x \in \R^d: \, (x,0) \in W_2, v_{r_j}(x,0) = 0\} \subset \{x \in \R^d: \, x_1 \in (-\delta_1,0]\}.
\end{equation}

By Lemma \ref{rough_estimate} there exists $C_4 =C_4(D,\lambda) > 0$ such that for sufficiently large $j$ and $x \in \partial W_1$ with $x_{d+1} > 0$ we have
\begin{equation}
\label{rough_estimate1}
v_{r_j}(x) \le \int_{(K_{D,\lambda}/r_j)^c} P(x_1-y_1,x_*-y_*,x_{d+1}) v_{r_j}(y_1,y_*,0) \, dy_1 \, dy_* +
C_4 r_j^{1/2} x_{d+1}.
\end{equation}

We next show that under an appropriate choice of $\eps_3$  we get for sufficiently large $j$
\begin{equation}
\label{boundary_estimate}
\text{$\tilde{Q}(x) \ge v_{r_j}(x)$ for $x \in \partial W_1$ and $\tilde{Q}(x) > v_{r_j}(x)$
for $x \in \partial W_1$ such that $v_{r_j}(x) > 0$.}
\end{equation}
We may assume that $x \in \partial W_1$ satisfy $x_{d+1} \ge 0$.

Let us consider
$$
U_1 = \{x \in \partial W_1: \, x_1 \ge - 2 \delta_1, \, |x_*| = R(\delta_1), x_{d+1} \ge 0\}.
$$
Note that for $x \in U_1$ we have $\kappa(x_*) = -4 \delta_1$ (with $\kappa$ defined as in \eqref{def:kappa-ast}).
Hence for $(x_1,x_*,x_{d+1}) \in U_1$ we get
\begin{eqnarray}
\nonumber
\tilde{Q}(x_1,x_*,x_{d+1}) &=& Q(x_1 - \delta_2,x_*,x_{d+1})\\
\nonumber 
&\ge& Q(x_1 - \delta_1,x_*,x_{d+1})\\
\nonumber
&\ge& (\beta  + \eps_2) \Psi(x_1-\delta_1 -\kappa(x_*),0,x_{d+1})\\
\label{U1_estimate}
&=& (\beta  + \eps_2) \Psi(x_1 + 3\delta_1,0,x_{d+1})
\end{eqnarray}
Now, $\eps_3$ must be chosen small enough so that for $x_{d+1} \in [0, N_1 \delta_1]$ we have $\eps_2 \Psi(3\delta_1,0,x_{d+1 }) > \eps_3$.
Then, for $x \in U_1$ (\ref{U1_estimate}) is bounded from below by
\begin{equation*}
\beta \Psi(x_1 ,0,x_{d+1}) + \eps_2 \Psi(3\delta_1,0,x_{d+1})  > \beta \Psi(x_1 ,x_*,x_{d+1}) + \eps_3
\ge v_{r_j}(x_1 ,x_*,x_{d+1}). 
\end{equation*}
So $\tilde{Q} > v_{r_j}$ on $U_1$.

Now, let us consider
$$
U_2 = \{x \in \partial W_1: \, x_1 = N_1 \delta_1, \,  x_{d+1} \ge 0\}.
$$
In order to show (\ref{boundary_estimate}) for $x \in U_2$ it is enough to choose $\eps_3$ and $\delta_2$ small enough. More precisely, we choose $\eps_3$ smaller if necessary, so that for $x_{d+1} \in [0,N_1\delta_1]$ we have $(\eps_2/2) \Psi((N_1 - 1)\delta_1,0,x_{d+1}) > \eps_3$. Moreover, we choose $\delta_2$ smaller if necessary, so that for $x_{d+1} \in [0,N_1\delta_1]$ we have $(\beta + \eps_2/2) \Psi(N_1\delta_1 - \delta_2,0,x_{d+1}) \ge \beta \Psi(N_1\delta_1,0,x_{d+1})$. Clearly, this is possible and then for $x \in U_2$ we have
\begin{eqnarray*}
\tilde{Q}(x_1,x_*,x_{d+1}) &=& Q(N_1\delta_1 - \delta_2,x_*,x_{d+1})\\
&\ge& (\beta  + \eps_2) \Psi(N_1\delta_1 - \delta_2,0,x_{d+1})\\
&\ge& (\beta + \eps_2/2) \Psi(N_1\delta_1 - \delta_2,0,x_{d+1}) + (\eps_2/2) \Psi((N_1 - 1)\delta_1,0,x_{d+1})\\
&>& \beta \Psi(N_1\delta_1,x_*,x_{d+1}) + \eps_3\\
&\ge& v_{r_j}(x_1 ,x_*,x_{d+1}) 
\end{eqnarray*}
Now, put
$$
U_3 = \{x \in \partial W_1: \, x_1 \le - 2 \delta_1,  x_{d+1} \in (0,|x_1|]\},
$$
$$
U_4 = \{x \in \partial W_1: \, x_1 \le - 2 \delta_1,  x_{d+1} \in (|x_1|,N_1 \delta_1)\},
$$
$$
U_5 = \{x \in \partial W_1: \, x_{d+1} = N_1 \delta_1\},
$$
$$
U_6 = \{x \in \partial W_1: \, x_1 \le - 2 \delta_1, x_{d+1} = 0\}.
$$
Next, by making $\delta_2$ even smaller if necessary, for any $x \in U_3 \cup U_4 \cup U_5$ and $y \in \R_+^d$ we have
$$
\frac{P(x_1-y_1-\delta_2,x_*-y_*,x_{d+1})}{P(x_1-y_1,x_*-y_*,x_{d+1})} > \frac{\beta + \eps_2/2}{\beta + \eps_2}
$$
Using this we get for $x \in U_3 \cup U_4 \cup U_5$
\begin{eqnarray*}
\tilde{Q}(x_1,x_*,x_{d+1}) &=& Q(x_1 - \delta_2,x_*,x_{d+1})\\
&\ge& (\beta  + \eps_2) \Psi(x_1 - \delta_2,0,x_{d+1})\\
&=& (\beta  + \eps_2) \int_{\R_+^d} P(x_1-y_1-\delta_2,x_*-y_*,x_{d+1}) y_1^{1/2} \, dy_1 \, dy_*\\
&>& (\beta  + \eps_2/2) \int_{\R_+^d} P(x_1-y_1,x_*-y_*,x_{d+1}) y_1^{1/2} \, dy_1 \, dy_*\\ 
&=& \text{I}.
\end{eqnarray*}

For $x \in U_3$ put $W_3 = (B_{|x_1|/2}^d(x))^c$ and for $x \in U_4 \cup U_5$ put $W_3 = \R^d$. Note that by (\ref{delta1_belt}) for $x = (x_1,x_*)$ satisfying $x_1 < -2 \delta_1$ and sufficiently large $j$ we have $v_{r_j}(y,0) = 0$ for $y \in B_{|x_1|/2}^d(x)$.
Using this, (\ref{beta_eps3}), (\ref{rough_estimate1}), and (\ref{general_estimate}) for sufficiently large $j$ and $x \in U_3 \cup U_4 \cup U_5$ we get
$$
v_{r_j}(x_1,x_*,x_{d+1}) \le \text{II} + \text{III} + \text{IV} + \text{V},
$$
where 
$$
\text{II} = \beta \int_{\R_+^d \cap W_2} P(x_1-y_1,x_*-y_*,x_{d+1}) y_1^{1/2} \, dy_1 \, dy_*,
$$
$$
\text{III} = \eps_3 \int_{W_3} P(x_1-y_1,x_*-y_*,x_{d+1}) \, dy_1 \, dy_*,
$$
$$
\text{IV} = \lambda \int_{(K_{D,\lambda}/r_j)^c \setminus W_2} P(x_1-y_1,x_*-y_*,x_{d+1}) |y|^{1/2} \, dy_1 \, dy_*
$$
and $\text{V} = C_4 r_j^{1/2} x_{d+1}$.

We next show that for sufficiently large $j$ and $x \in U_3 \cup U_4 \cup U_5$ we have
\begin{equation}
\label{U345}
\text{I} - \text{II} \ge \text{III} + \text{IV} + \text{V}.
\end{equation}

For $x \in U_3 \cup U_4 \cup U_5$ by Lemma \ref{Poisson1} we get
$$
\text{I} - \text{II} 
\ge \frac{\eps_2}{2} \int_{\R_+^d} P(x_1-y_1,x_*-y_*,x_{d+1}) y_1^{1/2} \, dy_1 \, dy_* \ge \frac{ C_1 \eps_2 x_{d+1}}{2 ((-x_1) \vee x_{d+1})^{1/2}}.
$$
Next, for $x \in U_3$ by Lemma \ref{Poisson3} we get $\text{III} \le \eps_3  C_3 x_{d+1}/((-x_1) \vee x_{d+1})$. On the other hand for $x \in U_4 \cup U_5$ we trivially have $\text{III} \le \eps_3 \le \eps_3 x_{d+1}/((-x_1)\vee x_{d+1})$. So for $x \in U_3 \cup U_4 \cup U_5$ we have
$$
\text{III} \le \frac{\eps_3 ( C_3 \vee 1 )x_{d+1}}{(-x_1)\vee x_{d+1}} 
\le \frac{\eps_3 ( C_3 \vee 1 )x_{d+1}}{((-x_1)\vee x_{d+1})^{1/2} (2 \delta_1)^{1/2}}.
$$
For $x \in U_3 \cup U_4 \cup U_5$ by Lemma \ref{Poisson2} we get
$$
\text{IV} \le \frac{ \lambda C_2 x_{d+1}}{(N_2 \wedge N_3)^{1/2} \delta_1^{1/2}}.
$$
Choosing $\epsilon_3$ smaller, if necessary, we obtain for $j$ big enough 
$$
\frac{1}{3} \left(\text{I} - \text{II} \right) \ge
\frac{ C_3 \eps_2 x_{d+1}}{6 ((-x_1) \vee x_{d+1})^{1/2}} \ge
\frac{\eps_3 ( C_3\vee 1) x_{d+1}}{((-x_1)\vee x_{d+1})^{1/2} (2 \delta_1)^{1/2}} \ge
\text{III}.
$$
By taking $j$ big enough, so that $r_j$ are small enough, we also get
$$
\frac{1}{3} \left(\text{I} - \text{II} \right) \ge
\frac{ C_1\eps_2 x_{d+1}}{6 N_1^{1/2} \delta_1^{1/2}} \ge
 C_4 r_j^{1/2} x_{d+1} =
\text{V}.
$$
By (\ref{N2N3}) we finally also get
$$
\frac{1}{3} \left(\text{I} - \text{II} \right) \ge
\frac{ C_1 \eps_2 x_{d+1}}{6 N_1^{1/2} \delta_1^{1/2}} \ge
\frac{ \lambda C_2 x_{d+1}}{(N_2 \wedge N_3)^{1/2} \delta_1^{1/2}} \ge
\text{IV}.
$$
This shows (\ref{U345}) for $x \in U_3 \cup U_4 \cup U_5$.

By (\ref{delta1_belt}) for $x \in U_6$ with  $x_1 \in (-4 \delta_1 + \delta_2, -2 \delta_1]$ we have $\tilde{Q}(x) > 0 = v_{r_j}(x)$ and for $x \in U_6$ with $x_1 \le -4 \delta_1 + \delta_2$ we have $\tilde{Q}(x) = v_{r_j}(x) = 0$, which finishes the proof of (\ref{boundary_estimate}).

\medskip

For a fixed, sufficiently small $\epsilon_3$ and corresponding large enough $j$, so that the above inequalities are satisfied, put 
\begin{equation*}
\tilde{w} = \left\{
\begin{aligned}
&\min(\tilde{Q}, v_{r_j})&\quad &\text{inside $W_1$,}\\
&v_{r_j}&\quad &\text{outside $W_1$}.\\
\end{aligned}
\right.
\end{equation*}

We have $\tilde{Q} \ge v_{r_j}$ on $\partial W_1$ so $\tilde{w} = v_{r_j}$ on $\partial W_1$, which implies that $\tilde{w}$ is continuous on $\R^{d+1}$. We have $\tilde{w} = v_{r_j} = r_j^{-1/2}$ on $\partial(r_j^{-1}D) \times \R$ and $0 \le \tilde{w} \le r_j^{-1/2}$ in $\R^{d+1}$. Let $\tilde{K} = \{x \in \R^d: \, \tilde{w}(x,0) = 0\}$. Clearly, $\tilde{K} = r_j^{-1} K_{D,\lambda} \cup \overline{B_1^d((-1+\delta_2)e_1^d)}$. Note that $v_{r_j}$ is harmonic on $((r_j^{-1}D) \times \R) \setminus ((r_j^{-1}K_{D,\lambda}) \times \{0\})$ and $\tilde{Q}$ is harmonic on $(\overline{B_1^d((-1+\delta_2)e_1^d)} \times \{0\})^c$.  By Lemma \ref{Q1 satisfies boundary condition} $\tilde{Q}(y,0) \le (\beta +\eps_1) (\dist(y,B_1^d((-1+\delta_2)e_1^d)))^{1/2}$ for any $y \in \R^d$.
By similar arguments as in the proof of Lemma \ref{lemmasupv} we have $\Delta \tilde{w} \le 0$ in $((r_j^{-1}D) \times \R) \setminus (\tilde{K} \times \{0\})$ in the viscosity sense and $\tilde{w}(y,0) \le \lambda (\dist(y,\tilde{K}))^{1/2}$ for $y \in \R^d$. For any $x \in \R^d$ and $y \in \R$ we clearly have $\tilde{w}(x,-y) = \tilde{w}(x,y)$. Furthermore, for any $x \in \R^d$ and $y_2 > y_1 \ge 0$ we have $\tilde{w}(x,y_2) \ge \tilde{w}(x,y_1)$.

Now put $w(x) = 1 - r_j^{1/2} \tilde{w}(x/r_j)$. Using the above properties of $\tilde{w}$ Lemma \ref{subharmonic} implies that $K_w:=\{x \in \R^d:\,w(\cdot,0)=1\} \in \calF(D,\lambda)$. On the other hand, $K_{D,\lambda} \subset K_w$ and $K_{D,\lambda} \ne K_w$ because 
$w(x,0) = 1$ for $x$ in some ball in $\R^d$ with center $0$. This gives contradiction with the maximal property of $u_{D,\lambda}$.
\end{proof}

\begin{proof}[Proof of Theorem \ref{thm_int_spectral}]
Let $u_{D,\lambda}$ be the function defined in Lemma \ref{lemma_uDlambda} and let $\varphi_{D,\lambda}$ be the function defined by $\varphi_{D,\lambda}(x) = u_{D,\lambda}(x,0)$ for $x \in \overline{D}$. We know that $K_{D,\lambda} \in \calF(D,\lambda)$ so for any $x \in D \setminus K_{D,\lambda}$ and $y \ge 0$ we have $u_{D,\lambda}(x,-y) = u_{D,\lambda}(x,y)$, which implies $\partial u_{D,\lambda}/\partial x_{d+1} (x,0) = 0$. Using this and \eqref{problemv} with $K=K_{D,\lambda}$ we get that $(-\Delta_D)^{1/2} \varphi_{D,\lambda}(x) = 0$ for $x \in D \setminus K_{D,\lambda}$. By Lemma \ref{C1boundary} the set $\{\varphi_{D,\lambda} = 1\} = K_{D,\lambda}$ is of class $C^1$. By Lemma \ref{normal_derivative} $\lim\limits_{t\to 0^+} \frac{\varphi_{D,\lambda}(\theta) - \varphi_{D,\lambda}(\theta+t\nu(\theta))}{\sqrt{t}} =\lambda$ for all $\theta\in \partial \{x \in \R^d: \, u_{D,\lambda}(x,0)  = 1\} = \partial K_{D,\lambda}$, where $\nu(\theta)$ denotes the exterior normal of $K_{D,\lambda}$ at $\theta$. It follows that $u = \varphi_{D,\lambda}$ satisfies Problem \ref{isbp}.

By the definition of $K_{D,\lambda}$ (see Lemma \ref{lemma_uDlambda}) the set $\{x \in \R^d: \, u_{D,\lambda}(x,0)  = 1\} = K_{D,\lambda}$ is convex.
\end{proof}

\begin{proof}[Proof of Proposition \ref{homolambdaprop}] 
Although it is almost straightforward, let us give a proof for the sake of the reader.
Let $K\in\calF(D_1,\lambda)$ and $v_1$ be the related subsolution. Then let $v_2$ be the solution of \eqref{problemv} with $D=D_2$. Then, by comparison, $v_2\geq v_1$, whence 
$$
\frac{1-v_2(y,0)}{\delta_K^{1/2}(y)} \le \frac{1-v_1(y,0)}{\delta_K^{1/2}(y)}\le \lambda\quad\text{for }y\in D_1\setminus K\,,
$$
while (see Lemma \ref{lemmadist})
$$
\frac{1-v_2(y,0)}{\delta_K^{1/2}(y)} \le \frac{1}{\delta_K^{1/2}(y)}\le \frac{1}{\text{dist}(\partial D_1,K)^{1/2}}\le\lambda\quad\text{for }y\in D_2\setminus D_1\,.
$$
Finally, since $D_2\setminus K=(D_1\setminus K)\cup(D_2\setminus D_1)$, we get $v_2\in\calF(D_2,\lambda)$ and this yields (i).
\medskip

Now let $K\in\calF(D,\lambda)$ and $v_K$ be the solution of \eqref{problemv}. Let $s>0$ and set
$$
v_s(x)=v_K(x/s)\quad\text{for }x\in sD \times \R\,.
$$
Then $v_s$ satisfies
$$
\left\{\begin{array}{ll}
\Delta v_s=0\quad&\text{in }( (sD)\times\R)\setminus((sK)\times\{0\})\\
v_s=0\quad&\text{on }\partial (sD)\times\R\\
v_s=1\quad&\text{in }(sK)\times\{0\}\\
\end{array}\right.
$$
and
$$ \frac{ |1-v_{s}(y,0)|}{\delta_{sK}^{1/2}(y)}=\frac{ |1-v_{K}(y/s,0)|}{s^{1/2}\delta_{K}^{1/2}(y/s)},
$$
whence $sK\in\calF(sD,s^{-1/2}\lambda)$. Viceversa, if $K'\in\calF(sD,\lambda')$ we can similarly see that $s^{-1}K'\in\calF(D,s^{1/2}\lambda')$. All together, we get
$s\calF(D,\lambda)=\calF(sD,s^{-1/2}\lambda)$ for every $\lambda>0$, which yields (ii).
\end{proof}
\begin{proof}[Proof of Proposition \ref{existlambda_S}]
As we have already said, the very definition of $\Lambda_S$ implies that no solution exists if $\lambda<\Lambda_S(D)$.
Now let $\lambda_n$ be a sequence such that $\calF(D,\lambda_n)\neq\emptyset$ and $\lim_{n\to+\infty} \lambda_n=\Lambda_S(D)$ and take any $\lambda>\Lambda_S(D)$. Then there exists $\tilde{n}$ such that $\lambda_{\tilde{n}}<\lambda$, whence $\calF(D,\lambda_{\tilde{n}})\subseteq\calF(D,\lambda)$ and a solution of Problem  \ref{isbp1} exists by Theorem \ref{thm_int_spectral}.
\\
Finally, we have to prove that a solution of Problem  \ref{isbp1} exists for $\lambda=\Lambda_S(D)$. First notice that we can choose the minimizing sequence $\lambda_n$ so that it is decreasing, then we have 
$\calF(D,\lambda_{n+1})\subseteq\calF(D,\lambda_n)$, whence
\begin{equation}\label{monotoneKn}
 K_{D,\lambda_{n+1}}\subseteq K_{D,\lambda_n}\,.
\end{equation}
For simplicity, let us shorten $K_{D,\lambda_n}$ into $K_n$ and set
$$
K=\lim_{n\to+\infty}K_n= \bigcap_{n=1}^\infty K_n\,.
$$
Furthermore, let us denote by $v_n$ the solution of 
\begin{equation}\label{problemvn}
\left\{\begin{array}{ll}
\Delta v_n=0\quad&\text{in }(D\times\R)\setminus(K_n\times\{0\})\\
v_n=0\quad&\text{on }\partial D\times\R\\
v_n=1\quad&\text{in }K_n\times\{0\}\\
\sup_{y\in D\setminus K_n} \dfrac{|v_{n}(y,0)-1|}{\delta_{K_n}^{1/2}(y)} = \lambda_n
\end{array}\right.
\end{equation}
given by Theorem \ref{thm_int_spectral}. By \eqref{monotoneKn} and comparison principle, we have
$$
v_{n+1}\leq v_n\quad\text{in }D\,.
$$
Then $v_n$ is a bounded decreasing sequence of harmonic functions, hence converging (uniformly in $D$) to a harmonic function $v_K$ which solves \eqref{problemv}.
Next we want to show that 
$$
\sup_{y\in D\setminus K} \dfrac{|v_K(y,0)-1|}{\delta_{K}^{1/2}(y)} \leq \Lambda_S(D)\,,
$$
or equivalently
\begin{equation}\label{conditionv}
1-v_K(y,0)\leq \Lambda_S(D)\, {\delta_{K}^{1/2}(y)}\quad\text{for }y\in D\setminus K\,.
\end{equation}
Now fix any point $y\in D\setminus K$, then there exist $\bar{n}$ such that $y\in D\setminus K_n$ for $n\geq\bar{n}$ and, by \eqref{problemvn}, it holds
$$
1-v_{n}(y,0)\leq \lambda_n\, {\delta_{K_n}^{1/2}(y)}\quad\text{for }n\geq\bar{n}\,.
$$
Passing to the limit as $n$ tends to $\infty$ and taking into account that $\lim_{n\to\infty}\delta_{K_n}(y)=\delta_K(y)$ we get \eqref{conditionv} as desired.
\\
Then $v$ is a subsolution of Problem \ref{isbp} for $\Lambda_S(D)$, that is $\calF(D,\Lambda_S(D))\neq\emptyset$ and a solution exists by Theorem  \ref{thm_int_spectral}.
\end{proof}

\begin{proof}[Proof of Proposition \ref{Bernoulli_constant}]
By scaling for any $r, s > 0$ we have $\Lambda_S(B_r^d(0)) = \sqrt{s} \Lambda_S(B_s^d(0))/\sqrt{r}$. We will estimate  $\Lambda_S(B_2^d(0))$.

Let $f: \R^d \to \R$ be the function defined by $f(x) = j(x - e_1^d)$, $x\in \R^d$ (where $j$ is given by (\ref{j_formula})). Note that $f$ is a radial function. Let $F$ be the harmonic extension of $f$. For $t > 0$, $x \in \R^{d+1}$ put $V_t(x) = t F(x) - (t - 1)$. Note that for any $t > 0$ $V_t \equiv 1$ on $B_1^d(0) \times \{0\}$. Fix $t > 0$ such that $V_t \equiv 0$ on $\partial B_2^d(0) \times \{0\}$. We have $V_t(2 e_1^d,0) = t f(2 e_1^d) - t + 1 = 0$, so $t = (1 - f(2 e_1^d))^{-1}$. For any $\tilde{x} \in B_1^d(0)$ we have 
$\frac{\partial}{\partial x_{d+1}} F(\tilde{x},0) \le 0$. Note that $F$ is harmonic on $(B_1^d(0) \times \{0\})^c$. Hence for any $\tilde{x} \in (\overline{B_1^d(0)})^c$ we have $\frac{\partial}{\partial x_{d+1}} F(\tilde{x},0) = 0$. Then by standard techniques we get
$$
\frac{\partial}{\partial x_{d+1}} F(\tilde{x},x_{d+1}) = 
\int_{\R^d} P((\tilde{x},x_{d+1}),y)  \left[\frac{\partial}{\partial x_{d+1}} F\right](y,0) \, dy \le 0,
$$
for any $\tilde{x} \in \R^d$, $x_{d+1} > 0$. This gives $\frac{\partial}{\partial x_{d+1}} V_t(\tilde{x},x_{d+1}) \le 0$ for any $\tilde{x} \in \R^d$, $x_{d+1} > 0$. Since $V_t \equiv 0$ on $\partial B_2^d(0) \times \{0\}$ we get $V_t(\tilde{x},x_{d+1}) \le 0$ for any $\tilde{x} \in \partial B_2^d(0)$, $x_{d+1} > 0$. By symmetry we get $V_t \le 0$ on $\partial B_2^d(0) \times \R$. 

Note that $\tilde{x} \to V_t(\tilde{x},0)$ is radial and radially nonincreasing. Hence for any $x_{d+1} \in \R$ $\tilde{x} \to V_t(\tilde{x},x_{d+1})$ is radial and radially nonincreasing. For $x \in \R^{d+1}$ put $\tilde{V}_t(x) = V_t(x)$ when $V_t(x) > 0$ and $\tilde{V}_t(x) = 0$ when $V_t(x) \le 0$. Put $\tilde{v}_t(x) = \tilde{V}_t(x,0)$ for $x \in \R^d$. For $x \in B_2^d(0)$ $\tilde{v}_t(x) = v_t(x) = t f(x) - t + 1$. Using this and Lemma \ref{q1_fractional_derivative} we get $\partial_{e_1^d}^{1/2} \tilde{v}_t(e_1^d) = - t  C_0$. Put $\lambda = t  C_0$. Note that $\{x\in \R^d\;:\; \tilde{v}_t(x)=1\} \in \calF(B_2^d(0),\lambda)$ by Lemma \ref{subharmonic} so $\Lambda_S{B_2^d(0)} \le \lambda$. Hence for any $r > 0$ we get $\Lambda_S(B_r^d(0)) = \sqrt{2} \Lambda_S(B_2^d(0))/\sqrt{r} \le \sqrt{2} t C_0/\sqrt{r}$.
\end{proof}

\section{The Brunn-Minkowski and Urysohn inequalities}

\begin{lemma}
\label{lemmalambdas}
Let  $D_0 \subset \Rd$ and  $D_1 \subset \Rd$ be bounded nonempty convex open sets, $s\in(0,1)$ and set $D_s=(1-s)D_0+sD_1$.
If $\calF(D_0,\lambda_0)$ and $\calF(D_1,\lambda_1)$ are not empty, then 
$$
K_s=(1-s) K_{D_0,\lambda_0} + s K_{D_1,\lambda_1} \in\calF(D_s,\max\{\lambda_0,\lambda_1\})\,.
$$
\end{lemma}

\begin{proof}
 Observe that $K_s$ are compact, $D_s$ are open, convex and $K_s\subseteq D_s$. For $i = 0, 1$ denote $v_i = v_{D_i,\lambda_i}$, $K_i = K_{D_i,\lambda_i}$.
Then, let $v_s^*$ be the function whose superlevel sets are the Minkowski combination of the corresponding superlevel sets of $v_0$ and $v_1$. More explicitly:
$$
v_s^*(x)=\sup\left\{\min(v_0(x_0),v_1(x_{1}))\,:\,x_0\in  \R^{d+1},\,x_1\in  \R^{d+1},\,(1-s)x_0+s x_1=x\right\}\,.
$$
Notice that 
\begin{equation}\label{propertyvs}
 v_s^*\in C( \R^{d+1})\,,\,\,\,0\leq v_s^*\leq 1\,,\,\,\,v_s^*=0\, \text{ on $ D_s^c \times \R$} and \,\{v_s^*=1\}=  K_s \times \{0\}\,.
 \end{equation}
Moreover, since $v_i\to 0$ for $|x|\to \infty$ for $i=0,1$, clearly also $v_s^*\to 0$ for $|x|\to \infty$.

Furthermore, it holds 
\begin{equation}\label{subharmonic-minkowski}
\Delta v_s^*\geq 0\,\,\,\text{ in the viscosity sense in }  (D_s\times\R) \setminus(K_s\times\{0\}) \,,
\end{equation}
see for instance the proof of Theorem 1 in \cite{CS2003-2}, in particular formula (38), or also \cite{BLS2009} (in particular formula (2.6) or Proposition 2.3, therein).

Now we want to show that $v_s^*$ satisfies \eqref{overdetermined}. Let $y\in D_s\setminus  K_s$ and fix the unique $x_s\in \partial K_s$ such that 
$$\delta_{K_s}(y)=\dist(y,K_s)=|y-x_s|\,,$$ 
which is possible, since $K_s$ is convex. 
Then, by the properties of Minkowski addition of convex sets, there exists $x_0\in\partial K_0$ and $x_1\in\partial K_1$ such that
$$
(1-s)x_0+s x_1=x_s\,\,\,\text{and}\,\,\,\nu_0(x_0)=\nu_1(x_1)=\nu_s(x_s)\,,
$$
where $\nu_i(x_i)$ denotes the outer unit normal of $K_i$ at $x_i$, for $i=0,1,s$.
Put $y_i=x_{i}+y-x_s$ for $i=0,1$. Then, by the convexity of the involved sets, we have
$$\delta_{K_0}(y_0)=|y_0-x_0|=\delta_{K_1}(y_1)=|y_1-x_1|=\delta_{K_s}(y)=|y-x_s|\,.
$$
Furthermore,
$$
(1-s)y_0+sy_1=y
$$
and thus
$$
v_s^*(y,0)\geq \min \{ v_0( y_0,0),v_1(y_1,0)\}.
$$
Without loss of generality, we may assume that the above minimum is attained at $v_1(y_1,0)$ and then we have
\begin{align*}
\frac{|v_s^*(y,0) - 1|}{\delta_{K_s}^{1/2}(y)}=\frac{ 1 - v_s^*(y,0) }{{\delta_{K_s}^{1/2}(y)}}\leq  \frac{ 1- v_1(y_1,0) }{\delta_{K_1}^{1/2}(y_1)}=\frac{|v(y_1,0) - 1|}{\delta_{K_1}^{1/2}(y_1)} \leq  \lambda_1 \leq \max\{\lambda_0,\lambda_1\}
\end{align*}
and thus $v_s^*$ satisfies \eqref{overdetermined}.

Now let us consider the solution $v_{K_s}$ of \eqref{problemv} associated to $K_s$: thanks to  \eqref{subharmonic-minkowski} we have $v_{K_s}\geq v_s^*$ in $D_s$ and since $v_{K_s}=v_s^*=1$ on $\partial K_s$, we have that 
\eqref{overdetermined} for $v_s^*$ easily implies \eqref{overdetermined} for $v_{K_s}$, which concludes the proof.
\end{proof}

Now we are ready to prove the Brunn-Minkowski inequality for $\Lambda_S$.
\begin{proof}[Proof of Theorem \ref{BMlambda}]
A straightforward consequence of the previous lemma is the following.
\begin{equation}\label{corlambdas}
\Lambda_S(D_s)\leq\max\{\Lambda_S(D_0),\,\Lambda_S(D_1)\}\,,
\end{equation}
which takes to \eqref{BMlambdaeq} thanks to a standard procedure based on the homogeneity of $\Lambda_S$. We give the proof for the sake of the reader.

Let $\tilde{D}_i=\Lambda_S(D_i)^2D_i$ for $i=0,1$ and observe that (ii) of Proposition \ref{homolambdaprop} yields
$$
\Lambda_S(\tilde{D}_0)=\Lambda_S(\tilde{D}_1)=1\,.
$$
Then, for every $\mu\in(0,1)$, by \eqref{corlambdas} we have
\begin{equation}\label{quasiBM}
\Lambda_S(\tilde{D}_\mu)\leq 1\,,
\end{equation}
where
$$
\tilde{D}_\mu=(1-\mu)\tilde{D}_0+\mu\tilde{D}_1\,.
$$
Now taking 
$$
\mu=\frac{s\Lambda_S(D_1)^{-2}}{(1-s)\Lambda_S(D_0)^{-2}+s\Lambda_S(D_1)^{-2}}\,,
$$
we have
$$
\tilde{D}_\mu=\frac{1}{(1-s)\Lambda_S(D_0)^{-2}+s\Lambda_S(D_1)^{-2}}\left((1-s){D}_0+s{D}_1\right)\,.
$$
Then \eqref{quasiBM} and (ii) of Proposition \ref{homolambdaprop} give
$$
\left[(1-s)\Lambda_S(D_0)^{-2}+s\Lambda_S(D_1)^{-2}\right]^{1/2}\Lambda_S\left((1-s){D}_0+s{D}_1\right)\leq 1\,,
$$
which coincides with \eqref{BMlambdaeq}.
\end{proof}

\begin{proof}[Proof of  Corollary \ref{urysohncor}]
The proof of the Urysohn inequality for the Bernoulli constant for the spectral half Laplacian is based on a standard procedure, as described in the proof of Corollary 2.2 of \cite{BS2009} or in Section 6 of \cite{S2015}. We leave the proof to the reader.
\end{proof}

\section{Appendix}

\begin{proof}[Proof of Lemma \ref{radial functions}]
In the following, $c$ denotes a positive constant depending only on $d$, whose value may change from line to line. The case $d=1$ is simple and we omit it. In the following let $d\geq 2$. Then, by translation into polar coordinates, denoting $S^d:=\partial B_1(0)$, we have
\begin{align*}
[u]_{1}&=\int_{0}^{\infty}\int_0^{\infty} \int_{S^d}\int_{S^d}\frac{(g(r)-g(\tau))^2r^{d-1}\tau^{d-1}}{|r\theta-\tau \phi|^{d+1}}\ d\theta\ d\phi \ d\tau dr\\
%&=\int_{0}^{\infty}\int_0^{\infty} (g(r)-g( \tau))^2r^{d-1}\tau^{d-1} \int_{S^d}\int_{S^d}\frac{1}{| r\theta-\tau\phi|^{d+1}}\ d\theta\ d\phi \ d\tau dr\\
%&=\int_{0}^{\infty}\int_0^{\infty} (g(r)-g( \tau))^2r^{d-1}\tau^{d-1} \int_{S^d}\int_{S^d}\frac{1}{(r^2+\tau^2-2r\tau\theta\cdot \phi)^{\frac{d+1}{2}}}\ d\theta\ d\phi \ d\tau dr\\
%&=c\int_{0}^{\infty}\int_0^{\infty} (g(r)-g( \tau))^2r^{d-1}\tau^{d-1}  \int_{S^d}\frac{1}{(r^2+\tau^2-2r\tau \phi_1)^{\frac{d+1}{2}}} \ d\phi \ d\tau dr\\
&=c\int_{0}^{\infty}\int_0^{\infty} (g(r)-g( \tau))^2r^{d-1}\tau^{d-1}  \int_{0}^{\pi}\frac{\sin(t)^{d-2}}{(r^2+\tau^2-2r\tau\cos(t))^{\frac{d+1}{2}}} \ dt \ d\tau dr.
\end{align*}
Using the monotonicity of $g$, we hence have
\begin{align*}
[u]_1&\geq \epsilon^2 c\int_0^{t_0}\int_{t_0}^{\infty}r^{d-1}\tau^{d-1}\int_{0}^{\pi}\frac{\sin(t)^{d-2}}{(r^2+\tau^2-2r\tau\cos(t))^{\frac{d+1}{2}}} \ dt \ d\tau dr\\
&=\epsilon^2t_0^{d-1}c \int_0^1\int_1^{\infty} r^{d-1}\tau^{d-1}\int_{0}^{\pi}\frac{\sin(t)^{d-2}}{(r^2+\tau^2-2r\tau\cos(t))^{\frac{d+1}{2}}} \ dt \ d\tau dr,
\end{align*}
where in the last step we substituted $r$ and $\tau$ with $t_0r$ and $t_0\tau$ abusing slightly the notation.
The claim follows once we show
\begin{equation}
\label{complicated integral}
\int_0^1\int_1^{\infty} r^{d-1}\tau^{d-1}\int_{0}^{\pi}\frac{\sin(t)^{d-2}}{(r^2+\tau^2-2r\tau\cos(t))^{\frac{d+1}{2}}} \ dt \ d\tau\ dr=\infty.
\end{equation}
Note here that we may write 
$$
\int_{0}^{\pi}\frac{\sin(t)^{d-2}}{(r^2+\tau^2-2r\tau\cos(t))^{\frac{d+1}{2}}} \ dt=r^{-d-1}\int_{0}^{\pi}\frac{\sin(t)^{d-2}}{(1+\gamma^2-2\gamma\cos(t))^{\frac{d+1}{2}}} \ dt
$$
with $\gamma=\frac{\tau}{r}>1$. Note that we have
\begin{equation}
\label{complicated bound}
\int_{0}^{\pi}\frac{\sin(t)^{d-2}}{(1+\gamma^2-2\gamma\cos(t))^{\frac{d+1}{2}}} \ dt\geq \frac{c}{\gamma^{d-2}(\gamma^2-1)^2}.
\end{equation}
The proof of this inequality is postponed to the end of this proof. With \eqref{complicated bound} we find
\begin{align*}
\int_0^1\int_1^{\infty} \int_{0}^{\pi}\frac{r^{d-1}\tau^{d-1}\sin(t)^{d-2}}{(r^2+\tau^2-2r\tau\cos(t))^{\frac{d+1}{2}}} \ dt \ d\tau\ dr&\geq 
%c\int_0^1\int_1^{\infty} r^{ d-1}\tau^{d-1}\frac{1}{\Big(\frac{\tau}{r}\Big)^{d-2}\Big(\Big(\frac{\tau}{r}\Big)^2-1\Big)^2}\ d\tau\ dr\\
%&=
c\int_0^1\int_1^{\infty}  r^{d-3} \frac{\tau}{ \left((\frac{\tau}{r})^2-1\right)^2}\ d\tau\ dr\\
%&=c\int_0^1\int_{1/r}^{\infty} r^{2d-3} \frac{rt}{ ( t^2-1)^2}r\ dt\ dr\\
%&=c\int_0^1 r^{2d-1} \int_{1/r}^{\infty}\frac{t}{ ( t^2-1)^2}\ dt\ dr\\
%&=c\int_0^1 r^{2d-1} \int_{r^{-2}-1}^{\infty}\frac{1}{ \tau^2}\ d\tau\ dr\\
%&=c\int_0^1 \frac{r^{2d-1}}{r^{-2}-1}\ dr
&=c\int_0^1 \frac{ r^{d+1}}{1-r^2}\ dr=\infty,
\end{align*}
which shows \eqref{complicated integral} and finishes the proof.

 To see \eqref{complicated bound}, we perform the same \textit{nonlinear change of variable} as in \cite[Appendix (A.26)]{B2016} by putting
\begin{equation}\label{change-of-variable}
\frac{\sin(t)}{\sqrt{1+\gamma^2-2\gamma \cos(t)}}=\frac{\sin(\theta)}{\gamma}.
\end{equation}
It follows by differentiating that
\begin{equation}\label{change-of-variable-part2}
dt=\Bigg(1-\frac{\cos(\theta)}{\sqrt{\gamma^2-\sin^2(\theta)}}\Bigg)\ d\theta.
\end{equation}

We also have $\cos^2(\theta) =\frac{(1-\gamma\cos(t))^2}{\gamma^2+1-2\gamma\cos(t)}$ and $\gamma^2-\sin^2(\theta)=\gamma^2\frac{(\gamma-\cos(t))^2}{\gamma^2+1-2\gamma\cos(t)}$. Hence $\sqrt{\gamma^2-\sin^2(\theta)}-\cos(\theta)=\sqrt{\gamma^2+1-2\gamma\cos(t)}$. Using this, \eqref{change-of-variable} and \eqref{change-of-variable-part2} we then have

\begin{align*}
\int_{0}^{\pi}&\frac{\sin(t)^{d-2}}{(1+\gamma^2-2\gamma \cos(t))^{\frac{d+1}{2}}}\ dt\\
&=\int_{0}^{\pi}\frac{\sin(\theta)^{d-2}}{\gamma^{d-2}(\sqrt{\gamma^2-\sin^2(\theta)}-\cos(\theta))^3}\Bigg(1-\frac{\cos(\theta)}{\sqrt{\gamma^2-\sin^2(\theta)}}\Bigg)\ d\theta\\
%&=\int_{0}^{\pi}\frac{\sin(\theta)^{d-2}}{\gamma^{d-2}(\sqrt{\gamma^2-\sin^2(\theta)}-\cos(\theta))^2\sqrt{\gamma^2-\sin^2(\theta)}} \ d\theta\\
%&=\int_{0}^{\pi}\frac{\sin(\theta)^{d-2}(\sqrt{\gamma^2-\sin^2(\theta)}+\cos(\theta))^2}{\gamma^{d-2}(\gamma^2-\sin^2(\theta)-\cos^2(\theta))^2\sqrt{\gamma^2-\sin^2(\theta)}} \ d\theta\\
&=\frac{1}{\gamma^{d-2}(\gamma^2-1)^2}\int_{0}^{\pi}\frac{\sin(\theta)^{d-2}(\gamma^2-\sin^2(\theta) +2\sqrt{\gamma^2-\sin^2(\theta)}\cos(\theta)+\cos^2(\theta))}{\sqrt{\gamma^2-\sin^2(\theta)}} \ d\theta\\
%&=\frac{1}{\gamma^{d-2}(\gamma^2-1)^2}\Bigg(\int_{0}^{\pi} \sin(\theta)^{d-2}\sqrt{\gamma^2-\sin^2(\theta)}\ d\theta +\int_{0}^{\pi}\frac{\sin(\theta)^{d-2}\cos^2(\theta)}{\sqrt{\gamma^2-\sin^2(\theta)}} \ d\theta\Bigg)\\
&=\frac{2}{\gamma^{d-2} (\gamma^2-1)^2}\Bigg(\int_{0}^{\frac{\pi}{2}} \sin(\theta)^{d-2}\sqrt{\gamma^2-\sin^2(\theta)}\ d\theta +\int_{0}^{\frac{\pi}{2}}\frac{\sin(\theta)^{d-2}(1-\sin^2(\theta))}{\sqrt{\gamma^2-\sin^2(\theta)}} \ d\theta\Bigg)
\end{align*}
using the symmetries of the integrals. It holds
\begin{align*}
\int_{0}^{\frac{\pi}{2}} \sin(\theta)^{d-2}(\gamma^2-\sin^2(\theta))^{\frac{1}{2}}\ d\theta
%&=\int_0^1 t^{d-2}(\gamma^2-t^2 )^{\frac{1}{2}}(1-t^2)^{-\frac{1}{2}}\ dt\\
&= \frac{1}{2}\int_0^1 \rho^{\frac{d-3}{2}}\Big(\frac{\gamma^2-\rho}{1-\rho}\Big)^{\frac{1}{2}} \ d\rho\qquad\text{and similarly}\\
\int_{0}^{\frac{\pi}{2}} \frac{\sin(\theta)^{d-2}(1-\sin^2(\theta))}{(\gamma^2-\sin^2(\theta))^{\frac{1}{2}}}\ d\theta&=\frac{1}{2}\int_0^1 \rho^{\frac{d-3}{2}}\Big(\frac{1-\rho }{\gamma^2-\rho}\Big)^{\frac{1}{2}}  \ d\rho.
\end{align*}
Hence, using $\gamma>1$
\begin{align*}
\int_{0}^{\pi}\frac{\sin(t)^{d-2}}{(1+\gamma^2-2\gamma \cos(t))^{\frac{d+1}{2}}}\ dt&=\frac{1}{\gamma^{d-2}(\gamma^2-1)^2}\int_0^{1}\rho^{\frac{d-3}{2}}\Bigg(\Big(\frac{\gamma^2-\rho}{1-\rho}\Big)^{\frac{1}{2}} +\Big(\frac{1-\rho }{\gamma^2-\rho}\Big)^{\frac{1}{2}} \Bigg)\ d\rho\\
&\geq \frac{1}{\gamma^{d-2}(\gamma^2-1)^2}\int_0^{1}\rho^{\frac{d-3}{2}} \ d\rho=\frac{c}{\gamma^{d-2}(\gamma^2-1)^2},
\end{align*}
which shows \eqref{complicated bound}.
\end{proof}

\end{document}